\documentclass[12pt]{amsart}
\usepackage{a4wide,amsfonts,graphicx,
amssymb,stmaryrd,amscd,amsmath,latexsym,amsbsy,xypic}
\xyoption{all}
\usepackage{marginnote}

\numberwithin{equation}{section}
\theoremstyle{plain}

\newtheorem{thm}{Theorem}[section]

\newtheorem{prop}[thm]{Proposition}

\newtheorem{defi}[thm]{Definition}
\newtheorem{lem}[thm]{Lemma}
\newtheorem{cor}[thm]{Corollary}

\theoremstyle{remark}
\newtheorem{rema}[thm]{Remark}

\newcommand{\Z}{\mathbb{Z}}

\newcommand{\C}{\mathbb{C}}

\title[]{On connection matrices of quantum Knizhnik-Zamolodchikov equations based on Lie super algebras}
\author{Wellington Galleas}
\author{Jasper V. Stokman}
\address{II. Institut f\"ur Theoretische Physik, Universit\"at Hamburg, Luruper Chaussee 149, 22761 Hamburg, Germany.}
\email{wellington.galleas@desy.de}
\thanks{The work of W.G. is supported by the German Science Foundation (DFG) under the Collaborative Research Center (SFB) 676: Particles, Strings and the Early Universe.}

\address{KdV Institute for Mathematics, University of Amsterdam,
Science Park 904, 1098 XH Amsterdam, The Netherlands.}
\email{j.v.stokman@uva.nl}
\subjclass[2000]{}
\begin{document}

\vspace{2em}
\keywords{}
%%%%%%%%%%%%%%%%%%%%%%%%%%%%%%%
\begin{abstract}
We propose a new method to compute connection matrices of quantum Knizhnik-Zamolodchikov equations associated to integrable
vertex models with super algebra and Hecke algebra symmetries. The scheme relies on decomposing the underlying spin representation of
the affine Hecke algebra in principal series modules and invoking the known solution of the connection problem for quantum affine Knizhnik-Zamolodchikov
equations associated to principal series modules. We apply the method to the spin representation underlying the 
$\mathcal{U}_q\bigl(\widehat{\mathfrak{gl}}(2|1)\bigr)$ Perk-Schultz model. We show that the corresponding connection matrices are described by an
elliptic solution of 
%a supersymmetric version of 
the dynamical quantum Yang-Baxter equation with spectral parameter. 
\end{abstract}
%%%%%%%%%%%%%%%%%%%%%%%%%%%%%%%
\maketitle
\begin{center}
{\it Dedicated to Masatoshi Noumi on the occasion of his 60th birthday}
\end{center}

\vspace{-10cm}
\begin{flushright}
{\footnotesize ZMP-HH/15-24}
\end{flushright}
\vspace{10cm}

\setcounter{tocdepth}{2}

%%%%%%%%%%%%%%%%%%%%%%%%%%%%%%%%%%%%%%%%%%
\section{Introduction}
%%%%%%%%%%%%%%%%%%%%%%%%%%%%%%%%%%%%%%%%%%%

\subsection{(Quantum) Knizhnik-Zamolodchikov equations}
Knizhnik-Zamolodchikov (KZ) equations were introduced in \cite{KZ} as a system of holonomic differential equations
satisfied by $n$-point correlation functions of primary fields in the Wess-Zumino-Novikov-Witten field theory \cite{WZ, N1, N2, W1, W2}.
Although they were introduced within a physical context, it has since proved to play an important role in several branches
of mathematics. One of the reasons for that lies in the fact that KZ equations exhibit strong
connections with the representation theory of affine Lie algebras. For instance, they are not restricted to Wess-Zumino-Novikov-Witten
theory and they can be used to describe correlation functions of general conformal field theories \cite{BPZ} associated with affine Lie 
algebras. Within the context of representation theory, correlation functions are encoded as matrix coefficients of intertwining
operators between certain representations of affine Lie algebras. This formulation is then responsible for associating important representation
theoretic information to the structure of the particular conformal field theory. Moreover, one remarkable feature of KZ equations 
from the representation theory point of view is related to properties of the monodromies (or connection matrices) of its solutions along closed paths. The 
latter was shown in \cite{K} to produce intertwining operators for quantum group tensor product representations.

The interplay between KZ equations and affine Lie algebras also paved the way for the derivation of a quantised version of
such equations having the representation theory of quantum affine algebras as its building block. In that case one finds 
a holonomic system of difference equations satisfied by matrix coefficients of a product of intertwining operators \cite{FR}.
The latter equations are known as quantum Knizhnik-Zamolodchikov equations, or qKZ equations for short. 

The fundamental ingredient for defining a qKZ equation is a solution of the quantum Yang-Baxter equation with spectral parameter, also referred to as a $R$-matrix.
Several methods have been developed along the years to find solutions of the Yang-Baxter equation; and among prominent examples
we have the \emph{Quantum Group} framework \cite{Ji1, Ji2, Ji3, D} and the \emph{Baxterization} method \cite{Jo}. These methods are not completely 
unrelated and solutions having $\mathcal{U}_q(\widehat{\mathfrak{gl}}(m|n))$ symmetry \cite{DGLZ} are known to be also obtained from
Baxterization of Hecke algebras \cite{CK, DA}. The particular cases $\mathcal{U}_q(\widehat{\mathfrak{gl}}(2))$ and 
$\mathcal{U}_q(\widehat{\mathfrak{gl}}(1|1))$ are in their turn obtained from the Baxterization of a quotient of the Hecke
algebra known as Temperley-Lieb algebra \cite{TL}. Other quantised Lie super algebras have also been considered within this program. 
Solutions based on the $\mathcal{U}_q({\widehat{\mathfrak{gl}}}^{(2)}(m|n))$, $\mathcal{U}_q(\widehat{\mathfrak{osp}}(m|n))$ and
$\mathcal{U}_q({\widehat{\mathfrak{osp}}}^{(2)}(m|n))$ have been presented in \cite{BS, GM1, GM2}. The latter cases also originate
from the Baxterization of Birman-Wenzl-Murakami algebras \cite{BW, Mu, GP, Gr1, Gr2}, as shown in \cite{GM2}. 

\subsection{Relation to integrable vertex models}
The quantum inverse scattering method attaches an integrable two-dimensional vertex model to an $R$-matrix.
A well known example is the six-vertex model, which is governed by the $R$-matrix obtained as the intertwiner $U(z_1)\otimes U(z_2)\rightarrow U(z_2)\otimes U(z_1)$ of 
$\mathcal{U}_q(\widehat{\mathfrak{sl}}(2))$-modules with $U(z)$ the $\mathcal{U}_q(\widehat{\mathfrak{sl}}(2))$ evaluation representation associated to the 
two-dimensional vector representation $U$ of $\mathcal{U}_q(\mathfrak{sl}(2))$. The 
qKZ equations associated to this $R$-matrix are solved by quantum correlation functions of the six-vertex model \cite{JM}. We will sometimes 
say that the qKZ equation is associated to the integrable vertex model governed by the $R$-matrix, instead of being associated to the $R$-matrix itself. 

A large literature has been devoted to the study of integrable systems based on
the Lie super algebra $\widehat{\mathfrak{gl}}(m|n)$, see for instance \cite{BB, BBO, dVL, Suz, Sa, EFS}.
The supersymmetric \textsc{t-j} model is one of the main examples. The associated $R$-matrix arises as
intertwiner of the Yangian algebra $\mathcal{Y}(\widehat{\mathfrak{gl}}(2|1))$. Another example is the $q$-deformed supersymmetric
\textsc{t-j} model \cite{Suz, Bar} whose $R$-matrix was firstly obtained by Perk and Schultz \cite{PS}. The relation between the Perk-Schultz
model and the $\mathcal{U}_q(\widehat{\mathfrak{gl}}(2|1))$ invariant $R$-matrix was clarified in \cite{Suz, MR, Bar}.

\subsection{Connection problems}
A basis of solutions of the qKZ equations can be constructed such that the solutions have asymptotically free behaviour deep in a particular
asymptotic sector $\mathcal{S}$.
The connection problem is the problem to explicitly compute the change of basis matrix between basis associated to different asymptotic sectors. 
The basis change matrix is then called a connection matrix. 

The connection problem for qKZ equations has been solved in special cases. Frenkel and Reshetikhin \cite{FR} solved it
for the qKZ equations attached to the six-vertex model.
Konno \cite{Ko} computed for a simple classical Lie algebra $\mathfrak{g}$ the connection matrices for the qKZ equations attached to 
the $\mathcal{U}_q(\widehat{\mathfrak{g}})$-intertwiner $U(z_1)\otimes U(z_2)\rightarrow U(z_2)\otimes U(z_1)$ with 
$U$ the vector representation of $\mathcal{U}_q(\mathfrak{g})$.
In both cases the computation of the connection matrices relies on explicitly solving the two-variable qKZ equation in terms of basic hypergeometric
series.

\subsection{The goals of the paper}
The aim of this paper is two-fold. Firstly we present a new approach to compute connection matrices of qKZ equations associated to intertwiners $R^W(z_1/z_2): W(z_1)\otimes W(z_2)\rightarrow
W(z_2)\otimes W(z_1)$ when the associated tensor product representation
$W(z_1)\otimes\cdots\otimes W(z_n)$ of evaluation modules, viewed as module over the {\it finite} quantum (super)group, becomes a Hecke algebra module by the action of the universal $R$-matrix on neighbouring tensor legs \cite{Ji1,Ji2}. Adding a quasi-cyclic operator, which physically is imposing quasi-periodic boundary conditions,
$W^{\otimes n}$ becomes a module over the affine Hecke algebra 
of type $A_{n-1}$, which we call the spin representation. The spin representation thus is governed by a constant $R$-matrix, which is the braid limit of the 
$R$-matrix $R^W(z)$ underlying the qKZ equations we started with. In this setup the qKZ equations coincide with Cherednik's \cite{C} quantum affine KZ equations associated to 
the spin representation.

The new approach 
is based on the solution of the connection problem of quantum affine KZ equations for principal series modules of the affine Hecke algebra, see \cite[\S 3]{S1} and the appendix
of the present paper. To compute the connection matrices of the qKZ equations associated to $R^W(z)$ it then suffices to decompose, 
if possible, the spin representation as direct sum of principal series modules and construct the connection matrices by glueing together the explicit connection matrices associated to the principal series blocks in the decomposition.

Secondly, we apply the aforementioned approach to compute the connection matrices for qKZ equations attached to the 
$\mathcal{U}_q(\widehat{\mathfrak{gl}}(2|1))$ Perk-Schultz model.
We show that they are governed by an explicit elliptic solution of 
%a supersymmetric version of 
the \emph{dynamical quantum Yang-Baxter equation}.
The latter equation was proposed by Felder \cite{F1} as the quantised version of a modified classical Yang-Baxter equation arising as the compatibility
condition of the Knizhnik-Zamolodchikov-Bernard equations \cite{B1, B2}. 

\subsection{Relation to elliptic face models}
Felder \cite{F1, F2} showed that solutions of the dynamical quantum
Yang-Baxter equation encodes statistical weights of face models. For instance, the solution of the dynamical quantum Yang-Baxter equation
arising from the connection matrices for the qKZ equations associated to the six-vertex model encodes the statistical weights of Baxter's \cite{Ba} eight-vertex face model \cite{FR,S1}.
More generally, for a simple Lie algebra $\mathfrak{g}$ of classical type $X_n$ and $U$
the vector representation of $\mathcal{U}_q(\mathfrak{g})$, Konno \cite{Ko} has shown that the connection matrices of the qKZ equations associated to the 
$\mathcal{U}(\widehat{\mathfrak{g}})$-intertwiner 
$U(z_1)\otimes U(z_2)\rightarrow U(z_2)\otimes U(z_1)$ 
are described by the statistical weights of the $X_n^{(1)}$ elliptic face models of Jimbo, Miwa and Okado \cite{JMO1,JMO2,JMO3}.

We expect that our elliptic solution of the 
%supersymmetric 
dynamical quantum Yang-Baxter equation, obtained from the connection matrices for the qKZ equations associated
to the $\mathcal{U}_q(\widehat{\mathfrak{gl}}(2|1))$ Perk-Schultz model, is closely related
to Okado's \cite{O} elliptic face model attached to $\mathfrak{gl}(2|1)$. 

\subsection{Future directions}
 It is natural to apply our techniques to compute connection matrices when the $R$-matrix is the $\mathcal{U}_q(\widehat{\mathfrak{gl}}(m|n))$-intertwiner $U(z_1)\otimes U(z_2)\rightarrow U(z_2)\otimes U(z_1)$ with $U$ the vector representation of the quantum super algebra $\mathcal{U}_q(\mathfrak{gl}(m|n))$, and to relate
 the connection matrices to Okado's \cite{O} elliptic face models attached to $\mathfrak{gl}(m|n)$. Another 
 natural open problem is the existence of a \emph{face-vertex} transformation \cite{Ba} turning our dynamical elliptic $R$-matrix
into an elliptic solution of the (non dynamical) quantum Yang-Baxter equation with spectral parameter. If such transformation exists it is natural to
expect that the resulting $R$-matrix will be an elliptic deformation of the $R$-matrix underlying the $\mathcal{U}_q(\widehat{\mathfrak{gl}}(2|1))$
Perk-Schultz model. Indeed, for $\mathfrak{gl}(2)$ it is well known that the connection matrices of the qKZ equations attached to the six-vertex 
model is governed by the elliptic solution of the dynamical quantum Yang-Baxter equation underlying Baxter's eight-vertex face model \cite{FR,S1}. 
By a face-vertex transformation, this dynamical $R$-matrix turns into the quantum $R$-matrix underlying Baxter's symmetric eight-vertex model, which 
can be regarded as the elliptic analogue of the six-vertex model.

We plan to return to these open problems in a future publication.

\vspace{1em}
\paragraph{{\bf Outline.}} This paper is organised as follows. In Section \ref{Sec2} we give the explicit elliptic solution of 
the dynamical quantum Yang-Baxter equation attached to the Lie super algebra $\mathfrak{gl}(2|1)$. 
In Section \ref{AHAsection} we discuss the relevant representation theory of the affine Hecke algebra. In \ref{Sec4} we present our new approach
to compute connection matrices of quantum affine KZ equations attached to spin representations.
In Section \ref{Sec5} we describe the spin representation associated to the $\mathcal{U}_q(\widehat{\mathfrak{gl}}(2|1))$ Perk-Schultz model and
decompose it as direct sum of principal series modules. The connection matrices of the quantum affine KZ equations associated to this spin representation is computed in 
Section \ref{Sec6}. In this section we also relate the connection matrices to the elliptic solution of the 
%supersymmetric 
dynamical quantum Yang-Baxter equation from Section \ref{Sec2}.
In Section 6 we need to have the explicit solution of the connection problem of quantum affine KZ equations associated to an arbitrary principal series module, while 
\cite[\S 3]{S1} only deals with a special class of principal series modules. We discuss the extension of the results from \cite[\S 3]{S1} to all principal series modules
in the appendix.

\vspace{1em}
\paragraph{{\it Acknowledgements.}} We thank Giovanni Felder and Huafeng Zhang for valuable comments and discussions.

%%%%%%%%%%%%%%%%%%%%%%%%%%%%%%%%%%%%%%%%%%%%%%%
\section{The elliptic solution of the dynamical quantum Yang-Baxter equation}\label{Sec2}

This paper explains how to obtain new elliptic dynamical $R$-matrices by solving connection problems for qKZ equations. 
The starting point is a constant $R$-matrix satisfying a Hecke relation.
We will describe a method
to explicitly compute the connection matrices of the qKZ equations associated to the Baxterization of the constant $R$-matrix.
In pertinent cases we show that these connection matrices are governed by explicit elliptic dynamical $R$-matrices.

We shall explain the technique in more detail from Section \ref{AHAsection} onwards. In this section we present the explicit elliptic dynamical $R$-matrix one obtains by applying this method to the spin representation of the affine Hecke algebra arising from the action of the universal $R$-matrix of the quantum group $\mathcal{U}_q(\mathfrak{gl}(2|1))$ on $V\otimes V$, with $V$ the ($3$-dimensional) vector representation of  $\mathcal{U}_q(\mathfrak{gl}(2|1))$.
%%%%%%%%%%%%%%%%%%%%%%%%%%%%%%%%%%%%%%%%%%%%%%%%%%%%%%%%%%%%%%%%%
\subsection{The Lie super algebra $\mathfrak{gl}(2|1)$}
%%%%%%%%%%%%%%%%%%%%%%%%%%%%%%%%%%%%%%%%%%%%%%%%%%%%%%%%%%%%%%%%%
Let $V = V_{\overline{0}} \oplus V_{\overline{1}}$ be a $\Z/2\Z$-graded vector space with even (bosonic) subspace $V_{\overline{0}}=\mathbb{C}v_1\oplus\mathbb{C}v_2$ and 
odd (fermionic) subspace $V_{\overline{1}}=\mathbb{C}v_3$. Let $p: \{1,2,3\}\rightarrow \Z/2\Z$ be the parity map
\begin{equation}\label{pmap}
p(i):=
\begin{cases}
\overline{0}\quad &\hbox{ if }\,\, i\in\{1,2\},\\
\overline{1}\quad &\hbox{ if }\,\, i=3,
\end{cases}
\end{equation}
so that $v_i\in V_{p(i)}$ for $i=1,2,3$.

Let $\mathfrak{gl}(V)$ be the associated Lie super algebra, with $\Z/2\Z$-grading given by
\begin{equation*}
\begin{split}
\mathfrak{gl}(V)_{\overline{0}}&=\{A\in\mathfrak{gl}(V) \,\, | \,\, A(V_{\overline{0}})\subseteq V_{\underline{0}}\,\, \& \,\, A(V_{\overline{1}})\subseteq V_{\overline{1}} \},\\
\mathfrak{gl}(V)_{\overline{1}}&=\{A\in\mathfrak{gl}(V) \,\, | \,\, A(V_{\overline{0}})\subseteq V_{\underline{1}}\,\, \& \,\, A(V_{\overline{1}})\subseteq V_{\overline{0}} \}
\end{split}
\end{equation*}
and with Lie super bracket $\lbrack X,Y\rbrack:=XY-(-1)^{\overline{X}\,\overline{Y}}YX$ for homogeneous elements $X, Y\in\mathfrak{gl}(V)$ of degree $\overline{X},\overline{Y}
\in\mathbb{Z}/2\Z$.
Note that $\mathfrak{gl}(V)\simeq\mathfrak{gl}(2|1)$ as Lie super algebras by identifying $\mathfrak{gl}(V)$ with a matrix Lie super algebra via the ordered basis $\{v_1,v_2,v_3\}$ of 
$V$. 

For $1\leq i,j\leq 3$ we write $E_{ij}\in \mathfrak{gl}(V)$ for the matrix units defined by
\[
E_{ij}(v_k):=\delta_{j,k}v_i,\qquad k=1,2,3.
\]
The standard Cartan subalgebra $\mathfrak{h}$ of the Lie super algebra $\mathfrak{gl}(V)$ is 
\[
\mathfrak{h}:=\mathbb{C}E_{11}\oplus\mathbb{C}E_{22}\oplus\mathbb{C}E_{33},
\]
which we endow with a symmetric bilinear $\bigl(\cdot,\cdot\bigr): \mathfrak{h}\times\mathfrak{h}\rightarrow\mathbb{C}$ by
\begin{equation*}
\bigl(E_{ii},E_{jj}\bigr)=
\begin{cases}
1\qquad &\hbox{ if }\,\, i=j\in\{1,2\},\\
-1\qquad &\hbox{ if }\,\, i=j=3,\\
0\qquad &\hbox{ otherwise}.
\end{cases}
\end{equation*}
In the definition of weights of a representation below we identify $\mathfrak{h}^*\simeq\mathfrak{h}$ via the non 
degenerate symmetric bilinear form $(\cdot,\cdot)$.

Let $W=W_{\overline{0}}\oplus W_{\overline{1}}$ be a finite dimensional
representation of the Lie super algebra $\mathfrak{gl}(V)$ with representation map $\pi: \mathfrak{gl}(V)\rightarrow
\mathfrak{gl}(W)$. We call $\lambda\in\mathfrak{h}$ a weight of $W$ if the weight space
\[
W[\lambda]:=\{u\in W \,\, | \,\, \pi(h)u=(h,\lambda)u\quad \forall\, h\in\mathfrak{h}\}
\]
is nonzero. We write $P(W)\subset\mathfrak{h}$ for the set of weights of $W$.

The vector representation of $\mathfrak{gl}(V)$ is the $\Z/2\Z$-graded vector space $V$, viewed as representation of the Lie super algebra $\mathfrak{gl}(V)$ by the
natural action of $\mathfrak{gl}(V)$ on $V$.
Note that $V$ decomposes as direct sum of weight spaces with the set of weights $P(V)=\{E_{11},E_{22},-E_{33}\}$
and weight spaces $V[E_{ii}]=\mathbb{C}v_i$ ($i=1,2$) and $V[-E_{33}]=\mathbb{C}v_3$. 

%%%%%%%%%%%%%%%%%%%%%%%%%%%%%%%%%%%%%%%%%
\subsection{The dynamical quantum Yang-Baxter equation associated to $\mathfrak{gl}(2|1)$}
%%%%%%%%%%%%%%%%%%%%%%%%%%%%%%%%%%%%%%%%%
We present here 
Felder's \cite{F1,F2} dynamical quantum Yang-Baxter equation for the
Lie super algebra $\mathfrak{gl}(2|1)$.

Let $W$ be a finite dimensional representation of $\mathfrak{gl}(V)$ with weight decomposition
\[
W=\bigoplus_{\lambda\in P(W)}W[\lambda]
\]
and suppose that $G(\mu): W^{\otimes n}\rightarrow W^{\otimes n}$ is a family of linear operators on $W^{\otimes n}$ depending meromorphically on $\mu\in\mathfrak{h}$.
For $\beta\in\mathbb{C}$ and $1\leq i\leq n$ we write 
\[G(\mu+\beta h_i): W^{\otimes n}\rightarrow W^{\otimes n}
\]
for the linear operator which acts as $G(\mu+\beta\lambda)$
on the subspace $W^{\otimes (i-1)}\otimes W[\lambda] \otimes W^{\otimes (n-i)}$ of $W^{\otimes n}$. More precisely, let 
$\textup{pr}^{(i)}_\lambda: W^{\otimes n}\rightarrow W^{\otimes n}$ be the projection onto the subspace $W^{\otimes (i-1)}\otimes W[\lambda] \otimes W^{\otimes (n-i)}$
along the direct sum decomposition
\[
W^{\otimes n}=\bigoplus_{\lambda\in P(W)}W^{\otimes (i-1)}\otimes W[\lambda] \otimes W^{\otimes (n-i)}.
\]
Then
\[
G(\mu+\beta h_i):=\sum_{\lambda\in P(W)} G(\mu+\beta\lambda)\circ\textup{pr}_\lambda^{(i)}.
\]
Let $\mathcal{R}^W(x;\mu): W\otimes W\rightarrow W\otimes W$ be linear operators, depending meromorphically on $x\in\mathbb{C}$
(the spectral parameter) and $\mu\in\mathfrak{h}$ (the dynamical parameters). Let $\kappa\in\mathbb{C}$. We say that 
$\mathcal{R}^W(x;\mu)$ satisfies the 
{\it dynamical quantum Yang-Baxter
equation in braid-like form} if 
\begin{equation}\label{dynqYBfirstW}
\begin{split}
\mathcal{R}_{12}^W(x;\mu+\kappa h_3)&\mathcal{R}_{23}^W(x+y;\mu-\kappa h_1)
\mathcal{R}_{12}^W(y;\mu+\kappa h_3)=\\
&=\mathcal{R}_{23}^W(y;\mu-\kappa h_1)\mathcal{R}_{12}^W(x+y;\mu+\kappa h_3)
\mathcal{R}_{23}^W(x;\mu-\kappa h_1)
\end{split}
\end{equation}
as linear operators on $W\otimes W\otimes W$. We say that $\mathcal{R}^W(x;\mu)$ is unitary if 
\[
\mathcal{R}^W(x;\mu)\mathcal{R}^W(-x;\mu)=\textup{Id}_{W^{\otimes 2}}.
\]
%%%%%%%%%%%%%%%%%%%%%%%%%%%%%%%%%%%%%%%%%%
\begin{rema}\label{rewrite}
Let $P\in\textup{End}(W\otimes W)$ be the permutation operator and write
\[
\check{\mathcal{R}}^W(x;\mu):=P\mathcal{R}^W(x;\mu) 
\]
with $\mathcal{R}^W$ satisfying \eqref{dynqYBfirstW}. Then $\check{\mathcal{R}}^W(x;\mu)$ satisfies the relation
\begin{equation}\label{dynqYBfirstW2}
\begin{split}
\check{\mathcal{R}}_{23}^W(x;\mu+\kappa h_1)&\check{\mathcal{R}}_{13}^W(x+y;\mu-\kappa h_2)
\check{\mathcal{R}}_{12}^W(y;\mu+\kappa h_3)=\\
&=\check{\mathcal{R}}_{12}^W(y;\mu-\kappa h_3)\check{\mathcal{R}}_{13}^W(x+y;\mu+\kappa h_2)
\check{\mathcal{R}}_{23}^W(x;\mu-\kappa h_1)
\end{split}
\end{equation}
which is the dynamical quantum Yang-Baxter equation as introduced by Felder \cite{F2} with dynamical shifts adjusted
to the action of the Cartan subalgebra $\mathfrak{h}$ of the Lie super algebra $\mathfrak{gl}(V)$. 
\end{rema}
%%%%%%%%%%%%%%%%%%%%%%%%%%%%%%%%%%%%%%%%%

%%%%%%%%%%%%%%%%%%%%%%%%%%%%%%%
\subsection{The dynamical $R$-matrix}\label{Rsub}
%%%%%%%%%%%%%%%%%%%%%%%%%%%%%%%%
We present an explicit elliptic solution of the dynamical quantum Yang-Baxter equation \eqref{dynqYBfirstW} for $W=V$ the vector representation.
Fix the nome $0<p<1$. We express the entries of the elliptic dynamical $R$-matrix in terms of products of renormalised Jacobi theta functions
\[
\theta(z_1,\ldots,z_r;p):=\prod_{j=1}^r\theta(z_j;p),\qquad \theta(z;p):=\prod_{m=0}^{\infty}(1-p^mz)(1-p^{m+1}/z).
\]
The natural building blocks of the $R$-matrix depend on the additional parameter $\kappa\in\mathbb{C}$ and are given by the functions
\begin{equation}\label{AB}
\begin{split}
A^y(x):&=\frac{\theta\bigl(p^{2\kappa},p^{y-x};p\bigr)}{\theta\bigl(p^y,p^{2\kappa-x};p\bigr)}p^{(2\kappa-y)x},\\
B^y(x)&:=\frac{\theta\bigl(p^{2\kappa-y},p^{-x};p\bigr)}{\theta\bigl(p^{2\kappa-x},p^{-y};p\bigr)}p^{2\kappa(x-y)} ,
\end{split}
\end{equation}
and the elliptic $c$-function 
\begin{equation}\label{C}
c(x):=p^{2\kappa x}\frac{\theta(p^{2\kappa+x};p)}{\theta(p^x;p)}.
\end{equation}
To write down explicitly the $R$-matrix $\mathcal{R}(x;\mu)=\mathcal{R}^V(x;\mu): V\otimes V\rightarrow V\otimes V$ 
it is convenient to identify $\mathfrak{h}\simeq\mathbb{C}^3$ via the ordered basis $(E_{11},E_{22},E_{33})$ of $\mathfrak{h}$,
\[
\phi_1E_{11}+\phi_2E_{22}+\phi_3E_{33}\leftrightarrow \underline{\phi}:=(\phi_1,\phi_2,\phi_3).
\]
Note that the weights $\{E_{11},E_{22},-E_{33}\}$ of $V$ correspond to $\{(1,0,0), (0,1,0), (0,0,-1)\}$.

Recall the parity map $p: \{1,2,3\}\rightarrow \Z/2\Z$ given by \eqref{pmap}.
%%%%%%%%%%%%%%%%%%%%%%%%%%%%%%%%%%%%%%%%%%
\begin{defi}\label{fmat}
We write $\mathcal{R}(x;\underline{\phi}): V\otimes V\rightarrow V\otimes V$ for the linear operator satisfying
\begin{equation*}
\begin{split}
\mathcal{R}(x,\underline{\phi})v_i\otimes v_i&=(-1)^{p(i)}\frac{c(x)}{c((-1)^{p(i)}x)}v_i\otimes v_i,\qquad\qquad\qquad\qquad\qquad\qquad 1\leq i\leq 3,\\
\mathcal{R}(x;\underline{\phi})v_i\otimes v_j&=A^{\phi_i-\phi_j}(x)v_i\otimes v_j +(-1)^{p(i)+p(j)}B^{\phi_i-\phi_j}(x)v_j\otimes v_i,\qquad 1\leq i\not=j\leq 3
\end{split}
\end{equation*}
with the $\kappa$-dependent coefficients given by \eqref{AB} and \eqref{C}.
\end{defi}
%%%%%%%%%%%%%%%%%%%%%%%%%%%%%%%%%%%%%%%%%%%%%%%
We can now state the main result of the present paper.
%%%%%%%%%%%%%%%%%%%%%%%%%%%%%%%%%%%%%%%%%%%%%%%%
\begin{thm}\label{mainTHMfirst}
The linear operator $\mathcal{R}(x;\underline{\phi})$ satisfies the 
dynamical
quantum Yang-Baxter equation in braid-like form
\begin{equation}\label{dynqYBfirst}
\begin{split}
\mathcal{R}_{12}(x;\underline{\phi}+\kappa h_3)&\mathcal{R}_{23}(x+y;\underline{\phi}-\kappa h_1)
\mathcal{R}_{12}(y;\underline{\phi}+\kappa h_3)=\\
&=\mathcal{R}_{23}(y;\underline{\phi}-\kappa h_1)\mathcal{R}_{12}(x+y;\underline{\phi}+\kappa h_3)
\mathcal{R}_{23}(x;\underline{\phi}-\kappa h_1)
\end{split}
\end{equation}
as linear operators on $V\otimes V\otimes V$, and the unitarity relation
\[
\mathcal{R}(x;\underline{\phi})\mathcal{R}(-x;\underline{\phi})=\textup{Id}_{V^{\otimes 2}}.
\]
\end{thm}
%%%%%%%%%%%%%%%%%%%%%%%%%%%%%%%%%%%%%%%%%%%%%%%%%%%%%%%%%%%%
The theorem can be proved by direct computations. The main point of the present paper is to explain how elliptic solutions of 
dynamical quantum Yang-Baxter equations,
like $\mathcal{R}(x;\underline{\phi})$, can be {\it found}
by explicitly computing connection matrices of quantum affine KZ equations. 

For example, 
%%%%%%%%%%%%%%%%%%%%%%%%%%%%%%%%%%%%%%%%%%%%%%%
\begin{equation}\label{ellRgl(2)}
\left(\begin{matrix} 1 & 0 & 0 & 0\\
0 & A^{y}(x) & B^{-y}(x) & 0\\
0 & B^y(x) & A^{-y}(x) & 0\\
0 & 0 & 0 & 1\end{matrix}\right)
\end{equation}
is an elliptic solution of a $\mathfrak{gl}(2)$ dynamical quantum Yang-Baxter equation in braid form, 
with $x$ the spectral parameter and $y$ the dynamical parameter, which governs the integrability of Baxter's $8$-vertex face model,
see for instance \cite{Ba, FR} and \cite{S1}. It was obtained in \cite{FR} 
by solving the connection problem of the qKZ equations associated to the spin-$\frac{1}{2}$ XXZ
chain. The associated spin representation is constructed from the $\mathcal{U}_q(\mathfrak{gl}(2))$ vector representation. 

In the following sections we show that our present solution $\mathcal{R}(x;\underline{\phi})$ can be obtained from the connection problem of the quantum affine KZ equations 
associated to the $\mathcal{U}_q(\widehat{\mathfrak{gl}}(2|1))$ Perk-Schultz model. In this case the associated spin representation is $V^{\otimes n}$
with $V$ the $\mathcal{U}_q(\mathfrak{gl}(2|1))$ vector representation, viewed as spin representation of the affine Hecke algebra by the action
of the universal $R$-matrix on neighbouring tensor legs \cite{Ji1,Ji2}. We expect that $\mathcal{R}(x;\underline{\phi})$ is closely related to Okado's \cite{O}
face model attached to $\mathfrak{sl}(2|1)$.

%%%%%%%%%%%%%%%%%%%%%%%%%%%%%%%%%%%%%%%%%%%%%%%%%
\begin{rema}
With respect to the ordered basis
\begin{equation}\label{orderedbasis}
\{v_1\otimes v_1, v_1\otimes v_2, v_1\otimes v_3, v_2\otimes v_1, v_2\otimes v_2, v_2\otimes v_3,
v_3\otimes v_1, v_3\otimes v_2, v_3\otimes v_3\},
\end{equation}
the solution $\mathcal{R}(x;\underline{\phi})$ is explicitly expressed as 
%%%%%%%%%%%%%%%%%%%%%%%%%%%%%%%%%%%%%%%%%%
\begin{equation*}
\resizebox{0.97\hsize}{!}{$\left(\begin{matrix} 1 & 0 & 0 & 0 & 0 & 0 & 0 & 0 & 0\\
0 & A^{\phi_1-\phi_2}(x) & 0 & B^{\phi_2-\phi_1}(x) & 0 & 0 & 0 & 0 & 0\\
0 & 0 & A^{\phi_1-\phi_3}(x) & 0 & 0 & 0 & -B^{\phi_3-\phi_1}(x) & 0 & 0\\
0 & B^{\phi_1-\phi_2}(x) & 0 & A^{\phi_2-\phi_1}(x) & 0 & 0 & 0 & 0 & 0\\
0 & 0 & 0 & 0 & 1 & 0 & 0 & 0 & 0\\
0 & 0 & 0 & 0 & 0 & A^{\phi_2-\phi_3}(x) & 0 & -B^{\phi_3-\phi_2}(x) & 0\\
0 & 0 & -B^{\phi_1-\phi_3}(x) & 0 & 0 & 0 & A^{\phi_3-\phi_1}(x) & 0 & 0\\
0 & 0 & 0 & 0 & 0 & -B^{\phi_2-\phi_3}(x) & 0 & A^{\phi_3-\phi_2}(x) & 0\\
0 & 0 & 0 & 0 & 0 & 0 & 0 & 0 & -\frac{c(x)}{c(-x)}
\end{matrix}\right) .$}
\end{equation*}
Note that the dependence on the dynamical parameters $\underline{\phi}$ is a $2$-dimensional dependence, reflecting the fact
that it indeed corresponds to the Lie super algebra $\mathfrak{sl}(2|1)$.
\end{rema}

%%%%%%%%%%%%%%%%%%%%%%%%%%%%%%%%%%%%%%%%%%
\section{Representations of the extended affine Hecke algebra}\label{AHAsection}
%%%%%%%%%%%%%%%%%%%%%%%%%%%%%%%%%%%%%%%%%%%
In this section we recall the relevant representation theoretic features of affine Hecke algebras.
%%%%%%%%%%%%%%%%%%%%%%%%%%%%%%%%%%%%%%%%%%%%
\subsection{The extended affine Hecke algebra}
%%%%%%%%%%%%%%%%%%%%%%%%%%%%%%%%%%%%%%%%%%%%
Let $n\geq 2$ and fix $0<p<1$ once and for all.
Fix a generic $\kappa\in\mathbb{C}$ and write $q=p^{-\kappa}\in\mathbb{C}^\times$. The extended affine Hecke algebra $H_n(q)$ of type
$A_{n-1}$ is the unital associative algebra over $\mathbb{C}$ generated by $T_1,\ldots,T_{n-1}$ and $\zeta^{\pm 1}$ with defining relations
\begin{equation*}
\begin{split}
T_iT_{i+1}T_i&=T_{i+1}T_iT_{i+1},\qquad 1\leq i<n-1,\\
T_iT_j&=T_jT_i,\quad\qquad\qquad |i-j|>1,\\
(T_i-q)&(T_i+q^{-1})=0,\qquad 1\leq i<n,\\
\zeta\zeta^{-1}&=1=\zeta^{-1}\zeta,\\
\zeta T_i&=T_{i+1}\zeta,\qquad\qquad 1\leq i<n-1,\\
\zeta^2T_{n-1}&=T_1\zeta^2.
\end{split}
\end{equation*}
Note that $T_i$ is invertible with inverse $T_i^{-1}=T_i-q+q^{-1}$. The subalgebra $H_n^{(0)}(q)$ of $H_n(q)$  generated by $T_1,\ldots,T_{n-1}$ is the finite Hecke algebra
of type $A_{n-1}$. Define for $1\leq i\leq n$,
\begin{equation}\label{Yi}
Y_i:=T_{i-1}^{-1}\cdots T_2^{-1}T_1^{-1}\zeta T_{n-1}\cdots T_{i+1}T_i\in H_n(q).
\end{equation}
Then $\lbrack Y_i,Y_j\rbrack=0$ for $1\leq i,j\leq n$ and $H_n(q)$ is generated as algebra by $H_n^{(0)}(q)$ and the abelian subalgebra $\mathcal{A}$ generated by
$Y_i^{\pm 1}$ ($1\leq i\leq n$).

The Hecke algebra $H_n^{(0)}(q)$ is a deformation of the group algebra of the symmetric group $S_n$ in $n$ letters. For $1\leq i<n$ we write $s_i$ for the standard Coxeter generator of $S_n$ given by the simple neighbour transposition $i\leftrightarrow i+1$. The extended affine algebra $H_n(q)$ is a deformation of the group algebra of the extended affine symmetric group $S_n\ltimes\mathbb{Z}^n$, where $S_n$ acts on $\mathbb{Z}^n$ via the permutation action. 

The commutation relations between $T_i$ ($1\leq i<n$) and $Y^\lambda:=Y_1^{\lambda_1}Y_2^{\lambda_2}\cdots Y_n^{\lambda_n}$ ($\lambda\in\mathbb{Z}^n$) are given
by the Bernstein-Zelevinsky cross relations
\begin{equation}\label{BZ}
T_iY^\lambda-Y^{s_i\lambda}T_i=(q-q^{-1})\left(\frac{Y^\lambda-Y^{s_i\lambda}}{1-Y_i^{-1}Y_{i+1}}\right).
\end{equation}
Note that the right hand side, expressed as element of the quotient field of $\mathcal{A}$, actually lies in $\mathcal{A}$.

%%%%%%%%%%%%%%%%%%%%%%%%%%%%%%%%%%%%%%%%%%%
\subsection{Principal series representations}\label{PSRsection}
%%%%%%%%%%%%%%%%%%%%%%%%%%%%%%%%%%%%%%%%%%%%
Let $I\subseteq\{1,\ldots,n-1\}$. We write $S_{n,I}\subseteq S_n$ for the subgroup generated by $s_i$ ($i\in I$). It is called the standard parabolic subgroup of $S_n$ 
associated to $I$. The standard parabolic subalgebra $H_I(q)$ of $H_n(q)$ associated to $I$ is 
the subalgebra generated by $T_i$ ($i\in I$) and 
$\mathcal{A}$. Note that $H_\emptyset(q)=\mathcal{A}$.

Let $\epsilon=(\epsilon_i)_{i\in I}$ be a $\#I$-tuple of signs, indexed by $I$, such that $\epsilon_i=\epsilon_j$ if $s_i$ and $s_j$ are in the same conjugacy class of $S_{n,I}$. 
Define
\[
E^{I,\epsilon}:=\{\gamma=(\gamma_1,\ldots,\gamma_n)\in \mathbb{C}^{n} \,\, | \,\, \gamma_i-\gamma_{i+1}=2\epsilon_i\kappa\quad \forall i\in I\}.
\]
For $\gamma\in E^{I,\epsilon}$ there exists a unique linear character $\chi_\gamma^{I,\epsilon}: H_I(q)\rightarrow\mathbb{C}$ satisfying
\begin{equation}\label{assignment}
\begin{split}
\chi_\gamma^{I,\epsilon}(T_i)&=\epsilon_iq^{\epsilon_i}=\epsilon_ip^{-\epsilon_i\kappa},\qquad i\in I,\\
\chi_\gamma^{I,\epsilon}(Y_j)&=p^{-\gamma_j},
\quad\qquad\qquad\,\,\,\,\,\, 1\leq j\leq n \, .
\end{split}
\end{equation}
Indeed, \eqref{assignment} respects the braid relations,
the Hecke relations $(T_i-q)(T_i+q^{-1})=0$ ($i\in I$) and the cross relations \eqref{BZ} for $i\in I$ and $1\leq j\leq n$.
We write $\mathbb{C}_{\chi_\gamma^{I,\epsilon}}$ for the corresponding one-dimensional $H_I(q)$-module. The {\it principal series module}
$M^{I,\epsilon}(\gamma)$ for $\gamma\in E^{I,\epsilon}$ is the induced $H_n(q)$-module
\[
M^{I,\epsilon}(\gamma):=\textup{Ind}_{H_I(q)}^{H_n(q)}\bigl(\chi_\gamma^{I,\epsilon}\bigr)=H_n(q)\otimes_{H_I(q)}\mathbb{C}_{\chi_\gamma^{I,\epsilon}}.
\]
We write $\pi_\gamma^{I,\epsilon}$ for the corresponding representation map and $v^{I,\epsilon}(\gamma):=1\otimes_{H_I(q)}1\in M^{I,\epsilon}(\gamma)$ for the canonical cyclic vector of $M^{I,\epsilon}(\gamma)$.

To describe a natural basis of the principal series module $M^{I,\epsilon}(\gamma)$ we need to recall first the definition of standard parabolic subgroups of $S_n$.
Let $w\in S_n$. We call an expression
\begin{equation}\label{redexp}
w=s_{i_1}s_{i_2}\cdots s_{i_{r}}
\end{equation}
reduced if the word \eqref{redexp} of $w$ as product of simple neighbour transpositions $s_i$ is of minimal length.
The minimal length $r$ of the word is called the length of $w$ and is denoted by $l(w)$. 
Let $S_n^I$ be the minimal coset representatives of the left coset space $S_n/S_{n,I}$.
It consists of the elements $w\in S_n$ such that $l(ws_i)=l(w)+1$ for all $i\in I$. 

For a reduced expression \eqref{redexp} of $w\in S_n$, the element
\[
T_w:=T_{i_1}T_{i_2}\cdots T_{i_{l(w)}}\in H_n^{(0)}(q)
\]
is well defined. 
Set
\[
v_w^{I,\epsilon}(\gamma):=\pi_\gamma^{I,\epsilon}(T_w)v^{I,\epsilon}(\gamma), \qquad w\in S_n^I.
\]
Then $\{v_w^{I,\epsilon}(\gamma)\}_{w\in S_n^I}$ is a linear basis of $M^{I,\epsilon}(\gamma)$ called the {\it standard basis} of $M^{I,\epsilon}(\gamma)$.
%%%%%%%%%%%%%%%%%%%%%%%%%%%%%%%%%%%%%%%%%%%%%%%
\subsection{Spin representations}
%%%%%%%%%%%%%%%%%%%%%%%%%%%%%%%%%%%%%%%%%%%%%%%
Let $W$ be a finite dimensional complex vector space and let $\mathcal{B}\in\textup{End}(W\otimes W)$ satisfy the braid relation
\[ \label{braid}
\mathcal{B}_{12}\mathcal{B}_{23}\mathcal{B}_{12}=\mathcal{B}_{23}\mathcal{B}_{12}\mathcal{B}_{23}
\]
as a linear endomorphism of $W^{\otimes 3}$ where we have used the usual tensor leg notation, i.e. $\mathcal{B}_{12}=\mathcal{B}\otimes\textup{Id}_W$ and
$\mathcal{B}_{23}=\textup{Id}_W\otimes\mathcal{B}$. In addition, let $\mathcal{B}$ satisfy the Hecke relation
\begin{equation}\label{HeckeRelation}
(\mathcal{B}-q)(\mathcal{B}+q^{-1})=0
\end{equation}
and suppose that $D\in\textup{GL}(W)$ is such that $\lbrack D\otimes D,\mathcal{B}\rbrack=0$. 
Then there exists a unique representation $\pi_{\mathcal{B},D}: H_n(q)\rightarrow\textup{End}(W^{\otimes n})$ such that
\begin{equation*}
\begin{split}
\pi_{\mathcal{B},D}(T_i)&:=\mathcal{B}_{i,i+1},\qquad 1\leq i<n,\\
\pi_{\mathcal{B},D}(\zeta)&:=P_{12}P_{23}\cdots P_{n-1,n}D_n
\end{split}
\end{equation*}
(recall that $P\in\textup{End}(W\otimes W)$ denotes the permutation operator).
We call $\pi_{\mathcal{B},D}$ the {\it spin representation}  associated to $(\mathcal{B},D)$.
Spin representations arise in the context of integrable one-dimensional spin chains with Hecke algebra symmetries 
and twisted boundary conditions, see for instance \cite{dV}. The corresponding spin chains are governed by the Baxterization
\begin{equation}\label{Baxterization}
R^{\mathcal{B}}(z):=P\circ\left(\frac{\mathcal{B}^{-1}-z\mathcal{B}}{q^{-1}-qz}\right)
\end{equation}
of $\mathcal{B}$, which is a unitary (i.e., $R^{\mathcal{B}}_{21}(z)^{-1}=R^{\mathcal{B}}(z^{-1})$) solution of the quantum Yang-Baxter equation
\begin{equation}\label{qYBB}
R_{12}^{\mathcal{B}}(x)R_{13}^{\mathcal{B}}(xy)R_{23}^{\mathcal{B}}(y)=
R_{23}^{\mathcal{B}}(y)R_{13}^{\mathcal{B}}(xy)R_{12}^{\mathcal{B}}(x).
\end{equation}

%%%%%%%%%%%%%%%%%%%%%%%%%%%%%%%%%%%%%%%%%%%%%
\subsection{Quantum KZ equations and the connection problem}\label{SecNew}
%%%%%%%%%%%%%%%%%%%%%%%%%%%%%%%%%%%%%%%%%%%%%
We introduce Cherednik's \cite{C} quantum KZ equations attached to representations of the affine Hecke algebra $H_n(q)$. Following
\cite{S1} we formulate the associated connection problem.

Let $\mathcal{M}$ be the field of meromorphic functions on $\mathbb{C}^n$.  
We write $F$ for the field of $\mathbb{Z}^n$-translation invariant meromorphic functions
on $\mathbb{C}^n$.

Let $\{e_i\}_{i=1}^n$ be the standard linear basis of $\mathbb{C}^n$,
with $e_i$ having a one at the $i$th entry and zeros everywhere else. We define an action $\sigma: S_n\ltimes\mathbb{Z}^n\rightarrow\textup{GL}(\mathcal{M})$
 of the extended affine symmetric group $S_n\ltimes\mathbb{Z}^n$ on $\mathcal{M}$ by
\begin{equation*}
\begin{split}
\bigl(\sigma(s_i)f\bigr)(\mathbf{z})&:=f(z_1,\ldots,z_{i-1},z_{i+1},z_i,z_{i+2},\ldots,z_n),\qquad 1\leq i<n,\\
\bigl(\sigma(\tau(e_j))f\bigr)(\mathbf{z})&:=f(z_1,\ldots,z_{j-1},z_j-1,z_{j+1},\ldots, z_{n}),\qquad 1\leq j\leq n
\end{split}
\end{equation*}
for $f\in\mathcal{M}$,
where we have written $\mathbf{z}=(z_1,\ldots,z_n)$ and $\tau(e_j)$ denotes the element in $S_n\ltimes\mathbb{Z}^n$ corresponding
to $e_j\in\mathbb{Z}^n$. 
The element $\xi:=s_1\cdots s_{n-2}s_{n-1}\tau(e_n)$ 
acts as $\bigl(\sigma(\xi)f\bigr)(\mathbf{z})=f(z_2,\ldots,z_{n},z_1-1)$.

Let $L$ be a finite dimensional complex vector space. We write
$\sigma_L$ for the action $\sigma\otimes\textup{Id}_L$ of $S_n\ltimes\mathbb{Z}^n$ on the corresponding space $\mathcal{M}\otimes L$ of meromorphic $L$-valued
functions on $\mathbb{C}^n$. 

Given a complex representation $(\pi,L)$ of $H_n(q)$ there exists a unique family $\{C_w^\pi\}_{w\in S_n\ltimes\mathbb{Z}^n}$ of $\textup{End}(L)$-valued meromorphic
functions $C_w^\pi$ on $\mathbb{C}^n$ satisfying the cocycle conditions
\begin{equation*}
\begin{split}
C_{uv}^\pi&=C_u^\pi\sigma_L(u)C_v^\pi\sigma_L(u^{-1}),\qquad u,v\in S_n\ltimes\mathbb{Z}^n,\\
C_e^\pi&\equiv \textup{Id}_L,
\end{split}
\end{equation*}
where $e\in S_n$ denotes the neutral element, and satisfying
\begin{equation*}
\begin{split}
C_{s_i}^\pi(\mathbf{z})&:=\frac{\pi(T_i^{-1})-p^{z_i-z_{i+1}}\pi(T_i)}{q^{-1}-qp^{z_i-z_{i+1}}},\qquad 1\leq i<n,\\
C_{\xi}^\pi(\mathbf{z})&:=\pi(\zeta).
\end{split}
\end{equation*}
It gives rise to a complex linear action $\nabla^\pi$ of $S_n\ltimes\mathbb{Z}^n$ on $\mathcal{M}\otimes L$ by
\[
\nabla^\pi(w):=C_w^\pi\sigma_L(w),\qquad w\in S_n\ltimes \mathbb{Z}^n.
\]
%%%%%%%%%%%%%%%%%%%%%%%%%%%%%%%%%%%%%%%%%%%%%
\begin{rema}
For a spin representation $\pi_{\mathcal{B},D}: H_n(q)\rightarrow\textup{End}(W^{\otimes n})$,
\[
C_{s_i}^{\pi_{\mathcal{B},D}}(\mathbf{z})=P_{i,i+1}R_{i,i+1}^{\mathcal{B}}(p^{z_i-z_{i+1}}),
\qquad 1\leq i<n
\]
with $R^{\mathcal{B}}(z)$ given by \eqref{Baxterization}.
\end{rema}
%%%%%%%%%%%%%%%%%%%%%%%%%%%%%%%%%%%%%%%%%%%%%
\begin{defi}
Let $(\pi,L)$ be a finite dimensional representation of $H_n(q)$. We say that $f\in\mathcal{M}\otimes L$ is a solution of the quantum affine KZ
equations if 
\[
C_{\tau(e_j)}^\pi(\mathbf{z})f(\mathbf{z}-e_j)=f(\mathbf{z}),\qquad j=1,\ldots,n.
\]
We write $\textup{Sol}^\pi\subseteq \mathcal{M}\otimes L$ for the solution space.
\end{defi}
%%%%%%%%%%%%%%%%%%%%%%%%%%%%%%%%%%%%%%%%%%%%%%
Alternatively,
\[
\textup{Sol}^{\pi}=\{f\in \mathcal{M}\otimes L \,\, | \,\, \nabla^\pi(\tau(\lambda))f=f\quad \forall\, \lambda\in\mathbb{Z}^n\}.
\]
Observe that $\textup{Sol}^\pi$ is a $F$-module. The symmetric group $S_n$ acts $F^{S_n}$-linearly on $\textup{Sol}^{\pi}$ by $\nabla^{\pi}|_{S_n}$.

In the limit $\Re(z_i-z_{i+1})\rightarrow-\infty$ ($1\leq i<n$) the transport operators $C_{\tau(\lambda)}^{\pi}(\mathbf{z})$ 
tend to commuting linear operators $\pi(\widetilde{Y}^\lambda)$ on $L$  for $\lambda\in\mathbb{Z}^n$. The commuting elements $\widetilde{Y}^\lambda\in H_n(q)$
are explicitly given by
\[
\widetilde{Y}^{\lambda}:=p^{-(\rho,\lambda)}T_{w_0}Y^{w_0\lambda}T_{w_0}^{-1}
\]
with $w_0\in S_n$ the longest Weyl group element and $\rho=((n-1)\kappa,(n-3)\kappa,\ldots,(1-n)\kappa)$ (see \cite{S1} and the appendix).

For a generic class of finite dimensional complex affine Hecke algebra modules the solution space of the quantum KZ equations can be
described explicitly in terms of asymptotically free solutions. The class of representations is defined as follows.
Write $\varpi_i=e_1+\cdots+e_i$ for $i=1,\ldots,n$.
%%%%%%%%%%%%%%%%%%%%%%%%%%%%%%%%%%%%%%%%%%%%%%%
\begin{defi} Let $\pi: H_n(q)\rightarrow\textup{End}(L)$ be a finite dimensional representation.
\begin{enumerate}
\item[{\bf 1.}] We call $(\pi,L)$ calibrated if $\pi(\widetilde{Y}_j)\in\textup{End}(L)$ is diagonalisable for $j=1,\ldots,n$, i.e. if
\[
L=\bigoplus_{\mathbf{s}}L[\mathbf{s}]
\]
with $L[\mathbf{s}]:=\{v\in L  \,\, | \,\, \pi(\widetilde{Y}^\lambda)v=p^{(\mathbf{s},\lambda)}v\,\,\, (\lambda\in\mathbb{Z}^n)\}$, 
where $\mathbf{s}\in\mathbb{C}^n/2\pi\sqrt{-1}\log(p)^{-1}\mathbb{Z}^n$.
\item[{\bf 2.}] We call $(\pi,L)$ generic if it is calibrated and if the nonresonance conditions
\[
p^{(\mathbf{s}^\prime-\mathbf{s},\varpi_i)}\not\in p^{\mathbb{Z}\setminus\{0\}}\qquad \forall\, i\in\{1,\ldots,n-1\}
\]
hold true for $\mathbf{s}$ and $\mathbf{s}^\prime$ such that $L[\mathbf{s}]\not=\{0\}\not=L[\mathbf{s}^\prime]$.
\end{enumerate}
\end{defi}
%%%%%%%%%%%%%%%%%%%%%%%%%%%%%%%%%%%%%%%%%%%%%%
Set 
\[
Q_+:=\bigoplus_{i=1}^{n-1}\mathbb{Z}_{\geq 0}(e_i-e_{i+1}).
\]
We recall the following key result on the structure of the solutions of the quantum KZ equations.
%%%%%%%%%%%%%%%%%%%%%%%%%%%%%%%%%%%%%%%%%%%%%%%
\begin{thm}[\cite{vMS}]\label{iso}
Let $(\pi,L)$ be a generic $H_n(q)$-representation and $v\in L[\mathbf{s}]$. There exists a unique 
meromorphic solution $\Phi_v^\pi$ of the quantum KZ equations characterised by the series expansion
\[
\Phi_v^\pi(\mathbf{z})=p^{(\mathbf{s},\mathbf{z})}\sum_{\alpha\in Q_+}\Gamma_v^\pi(\alpha)p^{-(\alpha,\mathbf{z})},\qquad
\Gamma_v^\pi(0)=v
\]
for $\Re(z_i-z_{i+1})\ll 0$ ($1\leq i<n$). The assignment $f\otimes v\mapsto f\Phi_v^\pi$ ($f\in F$, $v\in L[\mathbf{s}]$) defines
a $F$-linear isomorphism 
\[
S^\pi: F\otimes L\overset{\sim}{\longrightarrow}\textup{Sol}^{\pi}.
\]
\end{thm}
%%%%%%%%%%%%%%%%%%%%%%%%%%%%%%%%%%%%%%%%%%%%%%%%%
\begin{rema}\label{functoriality}
Let $(\pi,L)$ and $(\pi^\prime,L^\prime)$ be two generic $H_n(q)$-representations and 
$T: L\rightarrow L^\prime$ an intertwiner of $H_n(q)$-modules. Also write $T$ for its $\mathcal{M}$-linear
extension $\mathcal{M}\otimes L\rightarrow\mathcal{M}\otimes L$. Then
\[
T\circ S^\pi=S^{\pi^\prime}\circ T
\]
as $F$-linear maps $F\otimes L\rightarrow \textup{Sol}^{\pi^\prime}$ since
\[
T\bigl(\Phi_v^\pi(\mathbf{z})\bigr)=\Phi_{T(v)}^{\pi^\prime}(\mathbf{z})
\]
for $v\in L[\mathbf{s}]$.
\end{rema}
%%%%%%%%%%%%%%%%%%%%%%%%%%%%%%%%%%%%%%%%%%
Let $(\pi,L)$ be a generic $H_n(q)$-representation. For $w\in S_n$ we define a $F$-linear map
$\mathbb{M}^{\pi}(w): F\otimes L\rightarrow F\otimes L$ as follows,
\[
\mathbb{M}^\pi(w)=\bigl(\bigl(S^\pi\bigr)^{-1}\nabla^\pi(w)S^\pi\bigr)\circ \sigma_L(w^{-1}).
\]
The linear operators $\mathbb{M}^\pi(w)$ ($w\in S_n$) form a $S_n$-cocycle, called the {\it monodromy cocycle} of $(\pi,L)$. 

If $T: L\rightarrow L^\prime$ is a morphism between generic
$H_n(q)$-modules $(\pi,L)$ and $(\pi^\prime,L^\prime)$ then
\begin{equation}\label{functorialityM}
T\circ \mathbb{M}^\pi(w)=\mathbb{M}^{\pi^\prime}(w)\circ T,\qquad w\in S_n
\end{equation}
as $F$-linear maps $F\otimes L\rightarrow F\otimes L^\prime$ by Remark \ref{functoriality}.

 With respect to a choice of linear basis $\{v_i\}_i$ of $L$
consisting of common eigenvectors of the $\pi(\widetilde{Y}^\lambda)$ ($\lambda\in\mathbb{Z}^n$) we obtain from the monodromy 
cocycle matrices with coefficients in $F$, called connection matrices. The cocycle property implies braid-like relations for the connection matrices. 
We will analyse the connection matrices for the spin representation of $H_n(q)$ associated to the vector representation of
$\mathcal{U}_q(\mathfrak{gl}(2|1))$. It leads to explicit solutions of dynamical quantum Yang-Baxter equations.

%%%%%%%%%%%%%%%%%%%%%%%%%%%%%%%%%%%%%%%%%%%%%
\section{Connection matrices for principal series modules}\label{Sec4}
%%%%%%%%%%%%%%%%%%%%%%%%%%%%%%%%%%%%%%%%%%%%%%
First we recall the explicit form of the monodromy cocycle for a generic principal series module
$M^{I,\epsilon}(\gamma)$ ($\gamma\in E^{I,\epsilon}$), see \cite{S1} for the special case $\epsilon_i=+$ for all $i$ and the appendix for the general case.
Fix the normalised linear basis $\{\widetilde{b}_\sigma\}_{\sigma\in S_n^I}$ of $M^{I,\epsilon}(\gamma)$, given by \eqref{lc}, specialised to the $\textup{GL}_n$ root datum
and with $q$ in \eqref{lc} replaced by $p$. The basis elements are common eigenvectors for the action of $\widetilde{Y}^\lambda$  ($\lambda\in\mathbb{Z}^n$).
We write for $w\in S_n$,
\begin{equation}\label{mdefalt}
\mathbb{M}^{\pi_\gamma^{I,\epsilon}}(w)\widetilde{b}_{\tau_2}=\sum_{\tau_1\in S_n^I}m_{\tau_1\tau_2}^{I,\epsilon,w}(\mathbf{z};\gamma)
\widetilde{b}_{\tau_1}\qquad \forall\,\tau_2\in S_n^I
\end{equation}
with $m_{\tau_1\tau_2}^{I,\epsilon,w}(\mathbf{z};\gamma)\in F$ (as function of $\mathbf{z}$). For $w=s_i$ ($1\leq i<n$) the
coefficients are explicitly
given in terms of the elliptic functions $A^y(x), B^y(x)$ and the elliptic $c$-function
$c(x)$ (see \eqref{AB} and \eqref{C}) as follows.
\begin{enumerate}
\item[{\bf (a)}] On the diagonal, 
\begin{equation}\label{explicitdiagonal}
\begin{split}
m^{I,\epsilon,s_i}_{\sigma,\sigma}(\mathbf{z};\gamma)&=\epsilon_{i_\sigma}\frac{c(z_i-z_{i+1})}{c(\epsilon_{i_\sigma}(z_i-z_{i+1}))}
\,\qquad\qquad\qquad \hbox{ if }\quad \sigma\in S_n^I \,\, \& \,\, s_{n-i}\sigma\not\in S_n^I,\\
m^{I,\epsilon,s_i}_{\sigma,\sigma}(\mathbf{z};\gamma)&=A^{\gamma_{\sigma^{-1}(n-i)}-\gamma_{\sigma^{-1}(n-i+1)}}(z_i-z_{i+1}) \quad \,\,\hbox{ if } \quad \sigma\in S_n^I \,\, 
\&\,\, s_{n-i}\sigma\in S_n^I
\end{split}
\end{equation}
where,  if $\sigma\in S_n^I$ and $s_{n-i}\sigma\not\in S_n^I$, we write $i_\sigma\in I$ for the unique index such that $s_{n-i}\sigma=\sigma s_{i_\sigma}$.
\item[{\bf (b)}] All off-diagonal matrix entries are zero besides $m^{I,\epsilon,s_i}_{s_{n-i}\sigma,\sigma}(\mathbf{z};\gamma)$ with both $\sigma\in S_n^I$ and $s_{n-i}\sigma\in S_n^I$, which is given by
\begin{equation}\label{explicitoffdiagonal}
m^{I,\epsilon,s_i}_{s_{n-i}\sigma,\sigma}(\mathbf{z};\gamma)=B^{\gamma_{\sigma^{-1}(n-i)}-\gamma_{\sigma^{-1}(n-i+1)}}(z_i-z_{i+1}).
\end{equation}
\end{enumerate}
For $w\in S_n$ we write
\begin{equation}\label{Cexpl}
\mathbb{M}^{I,\epsilon,w}(\mathbf{z};\gamma)=
\bigl(m^{I,\epsilon,w}_{\sigma,\tau}(\mathbf{z};\gamma)\bigr)_{\sigma,\tau\in S_n^I} 
\end{equation}
for the matrix of $\mathbb{M}^{\pi_\gamma^{I,\epsilon}}(w)$ with respect to 
the $F$-linear basis $\{\widetilde{b}_\sigma\}_{\sigma\in S_n^I}$. 
The cocycle property of the monodromy cocycle then becomes
\[
\mathbb{M}^{I,\epsilon,ww^\prime}(\mathbf{z};\gamma)=\mathbb{M}^{I,\epsilon,w}(\mathbf{z};\gamma)
\mathbb{M}^{I,\epsilon,w^\prime}(w^{-1}\mathbf{z};\gamma) \qquad \forall\, w,w^\prime\in S_n
\]
and $\mathbb{M}^{I,\epsilon,e}(\mathbf{z};\gamma)=1$, where the symmetric group acts by permuting the variables $\mathbf{z}$. 
As a consequence, one directly obtains the following result.
%%%%%%%%%%%%%%%%%%%%%%%%%%%%%%%%%%%%%%%%%%%%%%%
\begin{prop}
The matrices \eqref{Cexpl} satisfy the braid type equations
\begin{equation}\label{YBconnection}
\begin{split}
\mathbb{M}^{I,\epsilon,s_i}(\mathbf{z};\gamma)\mathbb{M}^{I,\epsilon,s_{i+1}}(s_i\mathbf{z};\gamma)&\mathbb{M}^{I,\epsilon,s_i}(s_{i+1}s_i\mathbf{z};\gamma)=\\
&=\mathbb{M}^{I,\epsilon,s_{i+1}}(\mathbf{z};\gamma)\mathbb{M}^{I,\epsilon,s_i}(s_{i+1}\mathbf{z};\gamma)\mathbb{M}^{I,\epsilon,s_{i+1}}(s_is_{i+1}\mathbf{z};\gamma)
\end{split}
\end{equation}
for $1\leq i<n-1$ and the unitarity relation
\begin{equation}\label{unitarity}
\mathbb{M}^{I,\epsilon,s_i}(\mathbf{z};\gamma)\mathbb{M}^{I,\epsilon,s_i}(s_i\mathbf{z};\gamma)=1
\end{equation}
for $1\leq i<n$.
\end{prop}
%%%%%%%%%%%%%%%%%%%%%%%%%%%%%%%%%%%%%%%%%%%%%%%

In this paper we want to obtain explicit elliptic solutions of dynamical quantum Yang-Baxter equations
by computing connection matrices of a particular spin representation $\bigl(\pi_{\mathcal{B},D},W^{\otimes n}\bigr)$.
To relate $\mathbb{M}^{\pi_{\mathcal{B},D}}(s_i)$ to elliptic solutions of quantum dynamical Yang-Baxter equations acting locally on
the $i$th and $(i+1)$th tensor legs of $W^{\otimes n}$ one needs to compute the matrix coefficients of $M^{\pi_{\mathcal{B},D}}(s_i)$ with 
respect to a suitable tensor product basis $\{v_{i_1}\otimes\cdots\otimes v_{i_N}\}$ of $W^{\otimes n}$, where $\{v_i\}_i$ is some linear basis of $W$.

The approach is as follows. Suppose we have an explicit isomorphism of $H_n(q)$-modules 
\begin{equation}\label{decomposition}
T: W^{\otimes n}\overset{\sim}{\longrightarrow}\bigoplus_kM^{I^{(k)},\epsilon^{(k)}}(\gamma^{(k)}).
\end{equation}
Writing $\pi^{(k)}$ for the representation map of $M^{I^{(k)},\epsilon^{(k)}}(\gamma^{(k)})$ we conclude from
Remark \ref{functoriality} that the corresponding monodromy cocycles are related by
\begin{equation}\label{relspinprin}
\mathbb{M}^{\pi_{\mathcal{B},D}}(w)=T^{-1}\circ\Bigl(\bigoplus_{k}\mathbb{M}^{\pi^{(k)}}
(w)\Bigr)\circ T
\end{equation}
as $F$-linear endomorphisms of $F\otimes L$. If $\{v_i\}_i$ is a linear basis of $W$ and $\{v_{i_1}\otimes\cdots\otimes v_{i_N}\}$
the corresponding tensor product basis of $W^{\otimes n}$, then in general $T$ will not map the tensor product basis onto
the union (over $k$) of the linear bases $\{\widetilde{b}_\sigma\}_{\sigma\in I^{(k)}}$ of the constituents 
$M^{I^{(k)},\epsilon^{(k)}}(\gamma^{(k)})$ in \eqref{decomposition}. Thus trying to explicitly compute $\mathbb{M}^{\pi_{\mathcal{B},D}}(s_i)$ with
respect to a tensor product basis, using \eqref{relspinprin} and using the explicit form of  $\mathbb{M}^{\pi^{(k)}}(s_i)$ with respect to $\{\widetilde{b}_\sigma\}_{\sigma\in I^{(k)}}$,
will become cumbersome.

The way out is as follows. As soon as we know the existence of an isomorphism \eqref{decomposition} of $H_n(q)$-modules, we can
try to modify $T$ to obtain an explicit complex linear isomorphism
\[
\widetilde{T}: W^{\otimes n}
\overset{\sim}{\longrightarrow}\bigoplus_kM^{I^{(k)},\epsilon^{(k)}}(\gamma^{(k)})
\]
({\it not} an intertwiner of $H_n(q)$-modules!), which does have the property that a tensor product basis of $W^{\otimes n}$ is mapped to
the basis of the direct sum of principal series blocks consisting of the union of the bases
$\{\widetilde{b}_\sigma\}_{\sigma\in S_n^{I^{(k)}}}$. As soon as $\widetilde{T}$ is constructed, we 
can define the {\it modified monodromy cocycle}
$\{\widetilde{\mathbb{M}}^{\pi_{\mathcal{B},D}}(w)\}_{w\in S_n}$ of the spin representation $\pi_{\mathcal{B},D}$ by
\[
\widetilde{\mathbb{M}}^{\pi_{\mathcal{B},D}}(w):=\widetilde{T}^{-1}\circ\Bigl(\bigoplus_{k}
\mathbb{M}^{\pi^{(k)}}(w)\Bigr)\circ \widetilde{T},\qquad w\in S_n
\]
(clearly the $\widetilde{\mathbb{M}}^{\pi_{\mathcal{B},D}}(w)$ still form a $S_n$-cocycle).
Then the matrix of $\widetilde{\mathbb{M}}^{\pi_{\mathcal{B},D}}(s_i)$ with respect to the tensor product basis of $W^{\otimes n}$
will lead to an explicit solution of the dynamical quantum Yang-Baxter equation on $W\otimes W$ with spectral parameters.

We will apply this method for 
the spin representation associated to the $\mathcal{U}_q(\widehat{\mathfrak{gl}}(2|1))$ Perk-Schultz model in the next section.

%%%%%%%%%%%%%%%%%%%%%%%%%
\begin{rema}
The linear isomorphism $\widetilde{T}$ in the example treated in the next section is of the form
\[
\widetilde{T}=\bigl(\bigoplus_kG^{(k)}\bigr)\circ T
\]
with $T$ an isomorphism of $H_n(q)$-modules and with
$G^{(k)}$ the linear automorphism of $M^{I^{(k)},\epsilon^{(k)}}(\gamma^{(k)})$ mapping the standard basis element
$v_{\sigma}^{I^{(k)},\epsilon^{(k)}}(\gamma^{(k)})$ to a suitable constant multiple of $\widetilde{b}_\sigma$ for all $\sigma\in S_n^{I^{(k)}}$.
\end{rema}

%%%%%%%%%%%%%%%%%%%%%%%%%%%%%%%%%%%%%%%%%%%%%%%%%%%%%%%%%%%%%%%%%%%%%%%%
\section{The spin representation associated to $\mathfrak{gl}(2|1)$}\label{Sec5}
%%%%%%%%%%%%%%%%%%%%%%%%%%%%%%%%%%%%%%%%%%%%%%%%%%%%%%%%%%%%%%%%%%%%%%%%
Recall the vector representation $V=V_{\overline{0}}\oplus V_{\overline{1}}$ with $V_{\overline{0}}=\mathbb{C}v_1\oplus\mathbb{C}v_2$ and 
$V_{\overline{1}}=\mathbb{C}v_3$ of the Lie super algebra $\mathfrak{gl}(V)\simeq\mathfrak{gl}(2|1)$. The vector representation can be quantized, leading to 
the vector representation of the quantized universal enveloping algebra $\mathcal{U}_q(\mathfrak{gl}(2|1))$ on the same vector space $V$.
The action of the universal $R$-matrix of $\mathcal{U}_q(\mathfrak{gl}(2|1))$ on $V\otimes V$
gives rise to an explicit solution $\mathcal{B}: V\otimes V\rightarrow V\otimes V$ 
of the braid algebra relation \eqref{braid}. With respect to the ordered basis \eqref{orderedbasis} 
it is explicitly given by
\begin{equation} \label{sl21}
\mathcal{B}:=
\left(\begin{matrix} q & 0 & 0 & 0 & 0 & 0 & 0 & 0 & 0\\
0 & q-q^{-1} & 0 & 1 & 0 & 0 & 0 & 0 & 0\\
0 & 0 & q-q^{-1} & 0 & 0 & 0 & -1 & 0 & 0\\
0 & 1 & 0 & 0 & 0 & 0 & 0 & 0 & 0\\
0 & 0 & 0 & 0 & q & 0 & 0 & 0 & 0\\
0 & 0 & 0 & 0 & 0 & q-q^{-1} & 0 & -1 & 0\\
0 & 0 & -1 & 0 & 0 & 0 & 0 & 0 & 0\\
0 & 0 & 0 & 0 & 0 & -1 & 0 & 0 & 0\\
0 & 0 & 0 & 0 & 0 & 0 & 0 & 0 & -q^{-1}
\end{matrix}\right),
\end{equation}
see \cite{CK, DA} (we also refer the reader to \cite{KT, Y}). It satisfies the Hecke relation \eqref{HeckeRelation}. 

The Baxterization $R^{\mathcal{B}}(z)$ of $\mathcal{B}$ gives
\begin{eqnarray} 
\begin{split} \label{rmat}
R^{\mathcal{B}}(z)=
\left(\begin{matrix} a(z) & 0 & 0 & 0 & 0 & 0 & 0 & 0 & 0\\
0 & b(z) & 0 & c_+(z) & 0 & 0 & 0 & 0 & 0\\
0 & 0 & -b(z) & 0 & 0 & 0 & c_+(z) & 0 & 0\\
0 & c_-(z) & 0 & b(z) & 0 & 0 & 0 & 0 & 0\\
0 & 0 & 0 & 0 & a(z) & 0 & 0 & 0 & 0\\
0 & 0 & 0 & 0 & 0 & -b(z) & 0 & c_+(z) & 0\\
0 & 0 & c_-(z) & 0 & 0 & 0 & -b(z) & 0 & 0\\
0 & 0 & 0 & 0 & 0 & c_-(z) & 0 & -b(z) & 0\\
0 & 0 & 0 & 0 & 0 & 0 & 0 & 0 & w(z)
\end{matrix}\right) 
\end{split}
\end{eqnarray}
with 
\[
a(z):=\frac{q^{-1}-qz}{q^{-1}-qz},\quad b(z):=\frac{1-z}{q^{-1}-qz}, \quad c_+(z):=\frac{q^{-1}-q}{q^{-1}-qz}, \quad c_-(z):=\frac{(q^{-1}-q)z}{q^{-1}-qz} 
\]
and $w(z):=\frac{q^{-1}z-q}{q^{-1}-qz}$. Note that $P\circ \mathcal{B}$ can be re-obtained from $R^{\mathcal{B}}(z)$ by taking the braid limit $z\rightarrow\infty$.
The same solution $R^{\mathcal{B}}(z)$ is among the ones found by Perk and Schultz through the direct resolution of the quantum Yang-Baxter equation \eqref{qYBB}, see
\cite{PS}. It can also be obtained from the $\mathcal{U}_q(\widehat{\mathfrak{gl}}(2|1))$
invariant $R$-matrix with spectral parameter, see for instance \cite{CK, DA, KSU, BT, DGLZ, YZ}. 
Due to that the associated integrable vertex model is commonly refereed to as $\mathcal{U}_q(\widehat{\mathfrak{gl}}(2|1))$
Perk-Schultz model. Also, it is worth remarking that $R^{\mathcal{B}}(z)$ gives rise to 
a $q$-deformed version of the supersymmetric \textsc{t-j} model \cite{Suz, Bar}.

%%%%%%%%%%%%%%%%%%%%%%%%%%%%%%%%%%
\begin{rema}
We have written the $R$-matrix \eqref{rmat} in such a way that it satisfies the quantum Yang-Baxter
equation with standard tensor products. In order to make the gradation of $V$ explicitly manifested, and thus having
a solution of the graded Yang-Baxter equation as described in \cite{K}, we need to consider the matrix 
$\bar{R}^{\mathcal{B}} := P_g \, P \, R^{\mathcal{B}}$ where $P_g$ stands for the graded permutation operator.
\end{rema}
%%%%%%%%%%%%%%%%%%%%%%%%%%%%%%%%%%%%
\begin{rema}
Due to small differences of conventions and grading, one also needs to consider a simple gauge transformation in order to
compare \eqref{rmat} with the results presented in \cite{CK, DA, DGLZ, YZ}.
\end{rema}
%%%%%%%%%%%%%%%%%%%%%%%%%%%%%%%%%%%%

In order to proceed, we consider the linear map $D_{\underline{\phi}}: V\rightarrow V$ for a three-tuple $\underline{\phi}:=(\phi_1,\phi_2,\phi_3)$ of complex numbers defined by
\[
D_{\underline{\phi}}(v_i):=p^{-\phi_i}v_i,\qquad i=1,2,3 \, .
\]
It satisfies $\lbrack D_{\underline{\phi}}\otimes D_{\underline{\phi}}, \mathcal{B}\rbrack=0$. We write
$\pi_{\mathcal{B},\underline{\phi}}: H_n(q)\rightarrow\textup{End}(V)$ for the resulting spin representation 
$\pi_{\mathcal{B},D_{\underline{\phi}}}$. 

We are interested in computing the connection matrices of the quantum affine KZ equations associated to the spin representation $\pi_{\mathcal{B},\underline{\phi}}$.
With this goal in mind we firstly decompose the spin representation explicitly as direct sum of principal series modules. 

Let 
\[
\mathcal{K}_n:=\{\alpha:=(\alpha_1,\ldots,\alpha_n) \,\, | \,\, \alpha_i\in\{1,2,3\} \}
\]
and write 
$\mathbf{v}_\alpha:=v_{\alpha_1}\otimes v_{\alpha_2}\otimes\cdots\otimes v_{\alpha_n}\in V^{\otimes n}$ for $\alpha\in\mathcal{K}_n$.
We will refer to $\{\mathbf{v}_\alpha\}_{\alpha\in\mathcal{K}_n}$ as the {\it tensor product basis} of $V^{\otimes n}$. Next write 
\[
\mathcal{J}_n:=\{\mathbf{r}=(r_1,r_2,r_3)\in\mathbb{Z}_{\geq 0}^3 \,\, | \,\, r_1+r_2+r_3=n\}.
\]
Write $\mathcal{K}_n[\mathbf{r}]$ for the subset of $n$-tuples $\alpha\in\mathcal{K}_n$ with $r_j$ entries equal to $j$ for $j=1,2,3$. For instance,
\begin{equation}\label{alphar}
\alpha^{(\mathbf{r})}:=\bigl(\underbrace{3,\ldots,3}_{r_3},\underbrace{2,\ldots,2}_{r_2},\underbrace{1,\ldots,1}_{r_1}\bigr)\in\mathcal{K}_n[\mathbf{r}]
\end{equation}

Write $(V^{\otimes n})_{\mathbf{r}}:=\textup{span}\{\mathbf{v}_\alpha \,\, | \,\, \alpha\in\mathcal{K}_n[\mathbf{r}]\}$,
so that 
\[V^{\otimes n}=\bigoplus_{\mathbf{r}\in\mathcal{J}_n}(V^{\otimes n})_{\mathbf{r}}.
\]
%%%%%%%%%%%%%%%%%%%%%%%%%%%%%%%%%%%%%
\begin{lem}
$(V^{\otimes n})_{\mathbf{r}}$ is a $H_n(q)$-submodule of the spin representation $(\pi_{\mathcal{B},\underline{\phi}},V^{\otimes n})$.
\end{lem}
\begin{proof}
This follows immediately from the definition of the spin representation and the fact that $\lbrack D_{\underline{\theta}}\otimes D_{\underline{\theta}},\mathcal{B}\rbrack=0$
for all $\underline{\theta}\in\mathbb{C}^3$.
\end{proof}
%%%%%%%%%%%%%%%%%%%%%%%%%%%%%%%%%%%%%%%%%
The permutation action $\alpha\mapsto w\alpha$ of $S_n$ on 
$\mathcal{K}_n[\mathbf{r}]$, where $(w\alpha)_i:=\alpha_{w^{-1}(i)}$ for $1\leq i\leq n$, is transitive. The stabiliser subgroup of $\alpha^{(\mathbf{r})}$ (see \eqref{alphar})
is $S_{n,I^{(\mathbf{r})}}$ with $I^{(\mathbf{r})}\subseteq \{1,\ldots,n-1\}$ the subset 
\[
I^{(\mathbf{r})}:=\{1,\ldots,n-1\}\setminus \bigl(\{r_3,r_2+r_3\}\cap\{1,\ldots,n-1\}\bigr)
\]
For instance, if $1\leq r_3,r_2+r_3<n$ then 
\[
I^{(\mathbf{r})}=\{1,\ldots,r_3-1\}\cup\{r_3+1,\ldots,r_3+r_2-1\}\cup\{r_3+r_2+1,\ldots,n-1\},
\]
while $I^{(\mathbf{r})}=\{1,\ldots,n-1\}$ if $r_j=n$ for some $j$.
The assignment $w\mapsto w\alpha^{(\mathbf{r})}$ thus gives rise to a bijective map 
\[
\Sigma^{(\mathbf{r})}: S_n^{I^{(\mathbf{r})}}\overset{\sim}{\longrightarrow}\mathcal{K}_n[\mathbf{r}].
\]
Its inverse can be described as follows. For $\alpha\in\mathcal{K}_n[\mathbf{r}]$ and $j\in\{1,2,3\}$ write 
\[
1\leq k_1^{\alpha,(j)}<k_2^{\alpha,(j)}<\cdots<k_{r_j}^{\alpha,(j)}\leq n
\]
for indices $k$ such that $\alpha_{k}=j$ and denote
\begin{equation}\label{Inverse}
w_\alpha:=\left(\begin{matrix} 1 & \cdots & r_3 & r_3+1 & \cdots & r_3+r_2 & r_3+r_2+1 & \cdots & n\\
k_1^{\alpha,(3)} & \cdots & k_{r_3}^{\alpha, (3)} & k_{1}^{\alpha, (2)} & \cdots & k_{r_2}^{\alpha,(2)} & k_1^{\alpha,(1)} & \cdots & k_{r_1}^{\alpha, (1)}
\end{matrix}\right)\in S_n
\end{equation}
in standard symmetric group notations. Note that $w_\alpha\alpha^{(\mathbf{r})}=\alpha$. In addition, $w_\alpha\in S_n^{I^{(\mathbf{r})}}$ since 
\[
l(w_{\alpha}s_i)>
l(w_{\alpha})\qquad\forall\,  i\in I^{(\mathbf{r})},
\]
which is a direct consequence of the well known length formula
\begin{equation}\label{length}
l(w)=\# \{ (i,j) \,\, | \,\, 1\leq i<j\leq n \,\,\,\, \& \,\,\,\, w(i)>w(j)\}.
\end{equation}
It follows that 
\begin{equation*}
\bigl(\Sigma^{(\mathbf{r})}\bigr)^{-1}(\alpha)=w_\alpha,\qquad \forall\, \alpha\in\mathcal{K}_n[\mathbf{r}].
\end{equation*}

Let $\epsilon^{(\mathbf{r})}=
\{\epsilon^{(\mathbf{r})}_i\}_{i\in I^{(\mathbf{r})}}$
be given by
\begin{equation*}
\epsilon^{(\mathbf{r})}_i:=
\begin{cases}
-\quad &\hbox{ if } i<r_3 \\
+ \quad &\hbox{ else} \, ,
\end{cases}
\end{equation*}
and define $\gamma^{(\mathbf{r})}\in E^{I^{(\mathbf{r})},\epsilon^{(\mathbf{r})}}$ as
\begin{equation*}
\gamma_i^{(\mathbf{r})}:=
\begin{cases}
\eta_3^{(\mathbf{r})}+\phi_3+2i\kappa,\qquad &\hbox{ if }\,\, i\leq r_3,\\
\eta_2^{(\mathbf{r})}+\phi_2-2(i-r_3)\kappa,\qquad &\hbox{ if }\,\, r_3<i\leq r_3+r_2,\\
\eta_1^{(\mathbf{r})}+\phi_1-2(i-r_2-r_3)\kappa,\qquad &\hbox{ if }\,\, r_3+r_2<i\leq n \, ,
\end{cases}
\end{equation*}
with $\eta_j^{(\mathbf{r})}\in\mathbb{C}$ ($j=1,2,3$) given by
\begin{equation}\label{eta}
\begin{split}
\eta_1^{(\mathbf{r})}&:=-\pi\sqrt{-1}r_3\log(p)^{-1}+(r_1+1)\kappa,\\
\eta_2^{(\mathbf{r})}&:=-\pi\sqrt{-1}r_3\log(p)^{-1}+(r_2+1)\kappa,\\
\eta_3^{(\mathbf{r})}&:=-\pi\sqrt{-1}(n-1)\log(p)^{-1}-(r_3+1)\kappa.
\end{split}
\end{equation}

%%%%%%%%%%%%%%%%%%%%%%%%%%%%%%%%%%%%%%%%%%%%%%%%%%%
\begin{prop}\label{blockiso}
Let $\mathbf{r}\in\mathcal{J}_n$.
For generic parameters, there exists a unique isomorphism $\psi^{(\mathbf{r})}: M^{I^{(\mathbf{r})},\epsilon^{(\mathbf{r})}}(\gamma^{(\mathbf{r})})\overset{\sim}{\longrightarrow}
(V^{\otimes n})_{\mathbf{r}}$ of $H_n(q)$-modules 
mapping the cyclic vector
$v^{I^{(\mathbf{r})},\epsilon^{(\mathbf{r})}}(\gamma^{(\mathbf{r})})$ to
\[
\mathbf{v}_{\alpha^{(\mathbf{r})}}=v_3^{\otimes r_3}\otimes v_2^{\otimes r_2}\otimes v_1^{\otimes r_1}\in (V^{\otimes n})_{\mathbf{r}}.
\]
Furthermore, for $w\in S_n^{I^{(\mathbf{r})}}$ we have
\begin{equation}\label{standardtotensor}
\psi^{(\mathbf{r})}\bigl(v_w^{I^{(\mathbf{r})},\epsilon^{(\mathbf{r})}}(\gamma^{(\mathbf{r})})\bigr)=(-1)^{\eta(w)}\mathbf{v}_{w\alpha^{(\mathbf{r})}}
\end{equation}
with
\begin{equation}\label{etanew}
\eta(w):=\#\{(i,j) \,\, | \,\, 1\leq j<r_3<i\leq n \,\,\, \& \,\,\, w(i)<w(j)\}.
\end{equation}
\end{prop}
%%%%%%%%%%%%%%%%%%%%%%%%%%%%%%%%%%%%%%%%%%%%%%%%
\begin{proof}
From the explicit form of $\mathcal{B}$ it is clear that 
\begin{equation}\label{Teigenvalue}
\pi_{\mathcal{B},\underline{\phi}}(T_i)\mathbf{v}_{\alpha^{(\mathbf{r})}}=\mathcal{B}_{i,i+1}\mathbf{v}_{\alpha^{(\mathbf{r})}}=
\epsilon_i^{(\mathbf{r})}q^{\epsilon_i^{(\mathbf{r})}}\mathbf{v}_{\alpha^{(\mathbf{r})}}\qquad \forall\, i\in I^{(\mathbf{r})}.
\end{equation}
Next we show that $\pi_{\mathcal{B},\underline{\phi}}(Y_j)\mathbf{v}_{\alpha^{(\mathbf{r})}}=
p^{-\gamma^{(\mathbf{r})}_j}\mathbf{v}_{\alpha^{(\mathbf{r})}}$ for $1\leq j\leq n$. By the explicit expression of 
$\gamma^{(\mathbf{r})}$ the desired eigenvalues are 
\begin{equation*}
p^{-\gamma_i^{(\mathbf{r})}}=
\begin{cases}
(-1)^{n+1}q^{2i-r_3-1}p^{-\phi_3},\qquad &\hbox{ if }\,\, i\leq r_3,\\
(-1)^{r_3}q^{r_2+2(r_3-i)+1}p^{-\phi_2},\qquad &\hbox{ if }\,\, r_3<i\leq r_3+r_2,\\
(-1)^{r_3}q^{r_1+2(r_2+r_3-i)+1}p^{-\phi_1},\qquad &\hbox{ if }\,\, r_3+r_2<i\leq n.
\end{cases}
\end{equation*}
We give the detailed proof of the eigenvalue equation for $1\leq j\leq r_3$, the other two cases $r_3<j\leq r_3+r_2$ and $r_3+r_2<j\leq n$ can be verified by a similar
computation.

Since $j\leq r_3$, by \eqref{Yi} and \eqref{Teigenvalue} we have
\[
\pi_{\mathcal{B},\underline{\phi}}(Y_j)\mathbf{v}_{\alpha^{(\mathbf{r})}}=(-q^{-1})^{r_3-i}\pi_{\mathcal{B},\underline{\phi}}(T_{j-1}^{-1}\cdots T_1^{-1}\zeta 
T_{n-1}\cdots T_{r_3})\mathbf{v}_{\alpha^{(\mathbf{r})}}.
\]
Since $\mathcal{B}(v_3\otimes v_2)=-v_2\otimes v_3$ we get
\[
\pi_{\mathcal{B},\underline{\phi}}(Y_j)\mathbf{v}_{\alpha^{(\mathbf{r})}}=
(-1)^{r_2}(-q^{-1})^{r_3-i}\pi_{\mathcal{B},\underline{\phi}}(T_{j-1}^{-1}\cdots T_1^{-1}\zeta 
T_{n-1}\cdots T_{r_3+r_2})v_3^{\otimes (r_3-1)}\otimes v_2^{\otimes r_2}\otimes v_3\otimes v_1^{\otimes r_1}.
\] 
Then $\mathcal{B}(v_3\otimes v_1)=-v_1\otimes v_3$ gives
\begin{equation*}
\begin{split}
\pi_{\mathcal{B},\underline{\phi}}(Y_j)\mathbf{v}_{\alpha^{(\mathbf{r})}}&=
(-1)^{r_2+r_1}(-q^{-1})^{r_3-j}\pi_{\mathcal{B},\underline{\phi}}(T_{j-1}^{-1}\cdots T_1^{-1}\zeta)
v_3^{\otimes (r_3-1)}\otimes v_2^{\otimes r_2}\otimes v_1^{\otimes r_1}\otimes v_3\\
&=(-1)^{r_2+r_1}(-q^{-1})^{r_3-j}\phi_3\pi_{\mathcal{B},\underline{\phi}}(T_{j-1}^{-1}\cdots T_1^{-1})\mathbf{v}_{\alpha^{(\mathbf{r})}}.
\end{split}
\end{equation*}
Finally, using \eqref{Teigenvalue} again we find 
\[
\pi_{\mathcal{B},\underline{\phi}}(Y_j)\mathbf{v}_{\alpha^{(\mathbf{r})}}=
(-1)^{r_2+r_1}(-q^{-1})^{r_3-2j+1}\phi_3\mathbf{v}_{\alpha^{(\mathbf{r})}}=p^{-\gamma_j^{(\mathbf{r})}}\mathbf{v}_{\alpha^{(\mathbf{r})}}
\]
as desired. 

Consequently we have a unique surjective $H_n(q)$-intertwiner 
\[
\psi^{(\mathbf{r})}:
M^{I^{(\mathbf{r})},\epsilon^{(\mathbf{r})}}(\gamma^{(\mathbf{r})})\twoheadrightarrow
H_n(q)\mathbf{v}_{\alpha^{(r)}}\subseteq (V^{\otimes n})_{\mathbf{r}}
\]
mapping the cyclic vector $v^{I^{(\mathbf{r})},\epsilon^{(\mathbf{r})}}(\gamma^{(\mathbf{r})})$ to $\mathbf{v}_{\alpha^{(\mathbf{r})}}$. 
To complete the proof of the proposition, it thus suffices to prove \eqref{standardtotensor}.

Fix $\alpha\in\mathcal{K}_n[\mathbf{r}]$. We need to show that
\[
\psi^{(\mathbf{r})}\bigl(v_{w_\alpha}^{I^{(\mathbf{r})},\epsilon^{(\mathbf{r})}}(\gamma^{(\mathbf{r})})\bigr)=(-1)^{\eta(w_\alpha)}\mathbf{v}_{\alpha}.
\]
For the proof of this formula we first need to obtain a convenient reduced expression of $w_\alpha$.
Construct an element $w\in w_\alpha S_n^{I^{(\mathbf{r})}}$ (i.e., an element $w\in S_n$ satisfying $w\alpha^{(\mathbf{r})}=\alpha$) as product 
$w=s_{j_1}s_{j_2}\cdots s_{j_r}$ of simple neighbour transpositions such that, for all $u$, the $n$-tuple 
$s_{j_{u+1}}s_{j_{u+2}}\cdots s_{j_r}\alpha^{(\mathbf{r})}$ is of the form $(\beta_1^u,\ldots,\beta_n^u)$ with $\beta_{j_u}^u>\beta_{j_u+1}^u$. 
This can be done by transforming $\alpha^{(\mathbf{r})}$ to $\alpha$ by successive nearest neighbour exchanges between neighbours $(\beta,\beta^\prime)$
with $\beta>\beta^\prime$. Then it follows that
$l(w)=l\bigl(w_{\alpha}\bigr)$, hence $w=w_{\alpha}$. {}From this description of a reduced expression of $w_\alpha$ it follows that the number of pairs 
$(\beta_{j_u}^u,\beta_{j_{u}+1}^u)$ equal to $(t,s)$ is
$\#\{(i,j) \,\, | \,\, k_i^{\alpha,(s)}<k_j^{\alpha,(t)}\}$ for all $1\leq s<t\leq 3$. 

Since $\mathcal{B}v_3\otimes v_1=-v_1\otimes v_3$, $\mathcal{B}v_3\otimes v_2=-v_2\otimes v_3$ and $\mathcal{B}v_2\otimes v_1=v_1\otimes v_2$
we conclude that
\begin{equation*}
\begin{split}
\psi^{(\mathbf{r})}\bigl(v_{w_\alpha}^{I^{(\mathbf{r})},\epsilon^{(\mathbf{r})}}(\gamma^{(\mathbf{r})})\bigr)&=
\pi_{\mathcal{B},\underline{\phi}}(T_{w_\alpha})\mathbf{v}_{\alpha^{(\mathbf{r})}}\\
&=(-1)^{\eta(w_\alpha)}\mathbf{v}_{\alpha},
\end{split}
\end{equation*}
with $\eta(w)$ given by \eqref{etanew}.
 \end{proof}
%%%%%%%%%%%%%%%%%%%%%%%%%%%%%%%%%%%%%%%%%%%%%%%%%
\begin{lem}\label{technical}
Let $1\leq i<n$ and $\alpha\in\mathcal{K}_n[\mathbf{r}]$.
\begin{enumerate}
\item[{\bf (a)}] $s_{n-i}w_\alpha\in S_n^{I^{(\mathbf{r})}}$ if and only if $\alpha_{n-i}\not=\alpha_{n+1-i}$.
\item[{\bf (b)}] If $s_{n-i}w_\alpha\in S_n^{I^{(\mathbf{r})}}$ then $l(s_{n-i}w_\alpha)=l(w_\alpha)+1$ if and only if $\alpha_{n-i}>\alpha_{n+1-i}$.
\item[{\bf (c)}] If $s_{n-i}w_\alpha\not\in S_n^{I^{(\mathbf{r})}}$ then $i_{w_\alpha}\in \{1,\ldots,r_3-1\}$ if and only if $\alpha_{n-i}=3$ (recall that $i_{w_\alpha}\in I^{(\mathbf{r})}$
is the unique index such that $s_{n-i}w_\alpha=w_\alpha s_{i_{w_\alpha}}$).
\end{enumerate}
\end{lem}
%%%%%%%%%%%%%%%%%%%%%%%%%%%%%%%%%%%%%%%%%%%%%%%%%
\begin{proof}
The lemma follows directly from the explicit expression \eqref{Inverse} of $w_\alpha$ and the length formula \eqref{length}.
\end{proof}
%%%%%%%%%%%%%%%%%%%%%%%%%%%%%%%%%%%%%%%%%%%%%%%%%%
%%%%%%%%%%%%%%%%%%%%%%%%%%%%%%%%%%%%%%%%%%%%%%%%%%%%%%%%%%%%%%%%%%%%%%%%%%%%%%%%%%%%%%%%
\section{The elliptic $R$-matrix associated to $\mathfrak{gl}(2|1)$}\label{Sec6}
%%%%%%%%%%%%%%%%%%%%%%%%%%%%%%%%%%%%%%%%%%%%%%%%%%%%%%%%%%%%%%%%%%%%%%%%%%%%%%%%%%%%%%

%%%%%%%%%%%%%%%%%%%%%%%%%%%%%%%%%%%%%%%%%%%%%%%%%%
\subsection{The modified monodromy cocycle}
%%%%%%%%%%%%%%%%%%%%%%%%%%%%%%%%%%%%%%%%%%%%%%%%%%%
By Proposition \ref{blockiso} we have an isomorphism
\[
T:
V^{\otimes n}\overset{\sim}{\longrightarrow}
\bigoplus_{\mathbf{r}\in\mathcal{J}_n}M^{I^{(\mathbf{r})},\epsilon^{(\mathbf{r})}}(\gamma^{(\mathbf{r})})
\]
of $H_n(q)$-modules defined by
\[
T\bigl(\mathbf{v}_{w\alpha^{(\mathbf{r})}}\bigr)=(-1)^{\eta(w)}v_w^{I^{(\mathbf{r})},\epsilon^{(\mathbf{r})}}(\gamma^{(\mathbf{r})}),\qquad
\forall\, w\in S_n^{I^{(\mathbf{r})}},\, \forall\, \mathbf{r}\in\mathcal{J}_n.
\]
Write $\widetilde{b}_w^{(\mathbf{r})}$ for the basis $\widetilde{b}_w$ element of $M^{I^{(\mathbf{r})},\epsilon^{(\mathbf{r})}}(\gamma^{(\mathbf{r})})$
as defined in Section \ref{Sec4}, where $w\in S_n^{I^{(\mathbf{r})}}$. Let $G^{(\mathbf{r})}$ be the linear automorphism of $M^{I^{(\mathbf{r})},\epsilon^{(\mathbf{r})}}(\gamma^{(\mathbf{r})})$
defined by 
\[
G^{(\mathbf{r})}\bigl(\mathbf{v}_w^{I^{(\mathbf{r})},\epsilon^{(\mathbf{r})}}(\gamma^{(\mathbf{r})})\bigr)=\widetilde{b}_w^{(\mathbf{r})},\qquad \forall\, w\in S_n^{I^{(\mathbf{r})}}
\]
and write 
\[
\widetilde{T}:=\Bigl(\bigoplus_{\mathbf{r}\in\mathcal{J}_n}G^{(\mathbf{r})}\Bigr)\circ T:
V^{\otimes n}\longrightarrow
\bigoplus_{\mathbf{r}\in\mathcal{J}_n}M^{I^{(\mathbf{r})},\epsilon^{(\mathbf{r})}}(\gamma^{(\mathbf{r})}).
\]
$\widetilde{T}$ is a linear isomorphism given explicitly by
\[
\widetilde{T}\bigl(\mathbf{v}_{w\alpha^{(\mathbf{r})}}\bigr)=(-1)^{\eta(w)}\widetilde{b}_w^{(\mathbf{r})}, \qquad
\forall\, w\in S_n^{I^{(\mathbf{r})}},\, \forall\, \mathbf{r}\in\mathcal{J}_n.
\]
Set
\begin{equation}\label{reltildeM}
\widetilde{M}^{\pi_{\mathcal{B},D}}(u):=\widetilde{T}^{-1}\circ\Bigl(\bigoplus_{\mathbf{r}\in\mathcal{J}_n}M^{\pi^{(\mathbf{r})}}(u)\Bigr)\circ\widetilde{T}\in
\textup{End}\bigl(V^{\otimes n}\bigr)
\end{equation}
for $u\in S_n$, with $\pi^{(\mathbf{r})}$ the representation map of $M^{I^{(\mathbf{r})},\epsilon^{(\mathbf{r})}}(\gamma^{(\mathbf{r})})$. Then it follows that
\begin{equation}\label{explicitontensor}
\widetilde{M}^{\pi_{\mathcal{B},D}}(u)\mathbf{v}_\beta=
\sum_{\alpha\in\mathcal{K}_n[\mathbf{r}]}(-1)^{\eta(w_\alpha)+\eta(w_\beta)}
m_{w_{\alpha},w_{\beta}}^{I^{(\mathbf{r})},\epsilon^{(\mathbf{r})},u}(\mathbf{z};\gamma^{(\mathbf{r})})\mathbf{v}_\alpha\qquad \forall\, \beta\in\mathcal{K}_n[\mathbf{r}]
\end{equation}
for all $u\in S_n$.
Using the expressions of the connection coefficients $m_{w,w^\prime}^{I^{(\mathbf{r})},\epsilon^{(\mathbf{r})},s_i}(\mathbf{z};\gamma^{(\mathbf{r})})$ (see \eqref{explicitdiagonal} and \eqref{explicitoffdiagonal}) we obtain the following explicit formulas.
%%%%%%%%%%%%%%%%%%%%%%%%%%%%%%%%%%%%%%%%%%%%%%%%
\begin{cor}\label{explicitontensorCOR}
Let $\mathbf{r}\in\mathcal{J}_n$ and $1\leq i<n$. 
\begin{enumerate}
\item[{\bf (a)}] For $\beta=(\beta_1,\ldots,\beta_n)\in\mathcal{K}_n[\mathbf{r}]$ with  $\beta_{n-i}=\beta_{n+1-i}$ we have
\begin{equation*}
\widetilde{M}^{\pi_{\mathcal{B},D}}(s_i)\mathbf{v}_\beta=
\begin{cases}
\mathbf{v}_{\beta} \quad &\hbox{ if }\, \beta_{n-i}\in\{1,2\},\\
-\frac{c(z_i-z_{i+1})}{c(z_{i+1}-z_i)}\mathbf{v}_\beta \quad &\hbox{ if }\, \beta_{n-i}=3.
\end{cases}
\end{equation*}
\item[{\bf (b)}] For $\beta=(\beta_1,\ldots,\beta_n)\in\mathcal{K}_n[\mathbf{r}]$ with  $\beta_{n-i}\not=\beta_{n+1-i}$ we have
\begin{equation*}
\begin{split}
\widetilde{M}^{\pi_{\mathcal{B},D}}(s_i)\mathbf{v}_\beta&=
A^{\gamma_{w_\beta^{-1}(n-i)}^{(\mathbf{r})}-\gamma_{w_\beta^{-1}(n-i+1)}^{(\mathbf{r})}}(z_{i}-z_{i+1})\mathbf{v}_\beta\\
&+(-1)^{\delta_{\beta_{n-i},3}+\delta_{\beta_{n+1-i},3}}
B^{\gamma_{w_\beta^{-1}(n-i)}^{(\mathbf{r})}-\gamma_{w_\beta^{-1}(n-i+1)}^{(\mathbf{r})}}(z_{i}-z_{i+1})\mathbf{v}_{s_{n-i}\beta}.
\end{split}
\end{equation*}
\end{enumerate}
\end{cor}
%%%%%%%%%%%%%%%%%%%%%%%%%%%%%%%%%%%%%%%%%%%%%%%%
\begin{proof}
{\bf (a)} is immediate from the remarks preceding the corollary.\\
{\bf (b)} If $\beta_{n-i}\not=\beta_{n+1-i}$ then $s_{n-i}w_\beta\in S_n^{I^{(\mathbf{r})}}$ by Lemma \ref{technical}, hence $s_{n-i}w_\beta=w_\gamma$ for some
$\gamma\in\mathcal{K}_n[\mathbf{r}]$. Then
\[
\gamma=\Sigma^{(\mathbf{r})}(s_{n-i}w_\beta)=(s_{n-i}w_\beta)\alpha^{(\mathbf{r})}=s_{n-i}\beta,
\] 
hence $s_{n-i}w_\beta=w_{s_{n-i}\beta}$. 
Using the fact that 
\[
\eta(w_\beta)=\#\{(r,s)\,\, | \,\, k_r^{\beta,(2)}<k_s^{\beta,(3)}\}+\#\{(r,s) \,\, | \,\, k_r^{\beta,(1)}<k_s^{\beta,(3)}\}
\] 
we obtain
\[
(-1)^{\eta(w_\beta)+\eta(w_{s_{n-i}\beta})}=(-1)^{\delta_{\beta_{n-i},3}+\delta_{\beta_{n+1-i},3}}
\]
if $\beta_{n-i}\not=\beta_{n+1-i}$. 
The proof now follows directly from the explicit expressions \eqref{explicitdiagonal} and
\eqref{explicitoffdiagonal} of the connection coefficients.
\end{proof}
%%%%%%%%%%%%%%%%%%%%%%%%%%%%%%%%%%%%%%%%%%%%%%%%

%%%%%%%%%%%%%%%%%%%%%%%%%%%%%%%%%%%%%%%%%%%%%%
\subsection{Finding $\mathcal{R}(x;\underline{\phi})$}
%%%%%%%%%%%%%%%%%%%%%%%%%%%%%%%%%%%%%%%%%%%%%%%
In this subsection we fix $n=2$ and focus on computing the modified monodromy cocycle of the quantum affine KZ equations associated to the rank two spin representation 
$\pi_{\mathcal{B},\underline{\phi}}: H_2(q)\rightarrow \textup{End}(V^{\otimes 2})$. It will lead to the explicit expression of the elliptic dynamical $R$-matrix
$\mathcal{R}(x;\underline{\phi})$ from Subsection \ref{Rsub}.

From our previous results we know that the rank two spin representation $V^{\otimes 2}$ splits as $H_2(q)$-module into the direct sum of six principal series blocks
\begin{equation*}
\begin{split}
V^{\otimes 2}&=\bigoplus_{\mathbf{r}\in\mathcal{J}_2}(V^{\otimes 2})_{\mathbf{r}}\\
&=(V^{\otimes 2})_{(2,0,0)}\oplus (V^{\otimes 2})_{(0,2,0)}\oplus (V^{\otimes 2})_{(0,0,2)}
\oplus (V^{\otimes 2})_{(1,1,0)}\oplus (V^{\otimes 2})_{(1,0,1)}\oplus (V^{\otimes 2})_{(0,1,1)},
\end{split}
\end{equation*}
where the first three constituents are one-dimensional and the last three two-dimensional.
Write $s=s_1$ for the nontrivial element of $S_2$.

%%%%%%%%%%%%%%%%%%%%%%%%%%%%%%%%%%%%%%%%%%%%%%%%%%%%%%%%%%
\begin{lem}
For $n=2$ we have $\widetilde{M}^{\pi_{\mathcal{B},D}}(s)=\mathcal{R}(z_1-z_2;\underline{\phi})$ as linear operators
on $V\otimes V$.
\end{lem}
%%%%%%%%%%%%%%%%%%%%%%%%%%%%%%%%%%%%%%%%%%%%%%%%%%%%%%%%%%%%
\begin{proof}
This follows by a direct computation using Corollary \ref{explicitontensorCOR}. 
For instance, in the $9\times 9$-matrix representation of $\mathcal{R}(x;\underline{\phi})$ the first, 
fifth and ninth column of $R(x;\underline{\phi})$ arise from the action of $\widetilde{M}^{\pi_{\mathcal{B},D}}(s)$
on the one-dimensional constituents $(V^{\otimes 2})_{(2,0,0)}$, $(V^{\otimes 2})_{(0,2,0)}$ and 
$(V^{\otimes 2})_{(0,0,2)}$ respectively, in view of Corollary \ref{explicitontensorCOR}{\bf (a)}. The second and fourth columns correspond to the action of 
$\widetilde{M}^{\pi_{\mathcal{B},D}}(s)$
on $(V^{\otimes 2})_{(1,1,0)}=\textup{span}\{v_1\otimes v_2, v_2\otimes v_1\}$, in view of Corollary \ref{explicitontensorCOR}{\bf (b)}. 
Similarly, the third and seventh column (respectively
sixth and eighth column) corresponds to the action of $\widetilde{M}^{\pi_{\mathcal{B},D}}(s)$
on the constituent $(V^{\otimes 2})_{(1,0,1)}$ (respectively $(V^{\otimes 2})_{(0,1,1)}$).
\end{proof}
%%%%%%%%%%%%%%%%%%%%%%%%%%%%%%%%%%%%%%%%%%%%%%%%%%%%%%%%%%%%
\begin{cor}[Unitarity]
\[
\mathcal{R}(x;\underline{\phi})\mathcal{R}(-x;\underline{\phi})=\textup{Id}_{V^{\otimes 2}}.
\]
\end{cor}
%%%%%%%%%%%%%%%%%%%%%%%%%%%%%%%%%%%%%%%%%%%%%%%%%%%%%%%%%%%%
\begin{proof}
This follows from \eqref{unitarity} and \eqref{reltildeM}.
\end{proof}
%%%%%%%%%%%%%%%%%%%%%%%%%%%%%%%%%%%%%%%%%%%%%%%%%%%%%%%%%%%%%

%%%%%%%%%%%%%%%%%%%%%%%%%%%%%%%%%%%%%%%%%%%%%%%%%
\subsection{The dynamical quantum Yang-Baxter equation}
%%%%%%%%%%%%%%%%%%%%%%%%%%%%%%%%%%%%%%%%%%%%%%%%%

Next we prove that $\mathcal{R}(x;\underline{\phi})$ satisfies the 
dynamical quantum Yang-Baxter equation in braid-like form (see Theorem \ref{mainTHMfirst})
by computing the modified monodromy cocycle of the quantum affine KZ equations associated to the spin representation $\pi_{\mathcal{B},\underline{\phi}}: H_3(q)\rightarrow
\textup{End}(V^{\otimes 3})$ and expressing them in terms of local actions of $\mathcal{R}(x;\underline{\phi})$. So in this subsection, we fix $n=3$.

Let $\underline{\Psi}^{(j)}\in\mathbb{C}^3$ for $j=1,2,3$ and let $Q(\underline{\phi}): V^{\otimes 3}\rightarrow V^{\otimes 3}$ be a family of linear operators on $V^{\otimes 3}$ depending on $\underline{\phi}\in\mathbb{C}^3$.
We use the notation $Q(\underline{\phi}+\widehat{\underline{\Psi}}_i)$ to denote the linear operator on $V^{\otimes 3}$ which acts on the subspace $V^{\otimes (i-1)}\otimes\mathbb{C}v_j\otimes V^{\otimes (3-i)}$ as $Q(\underline{\phi}+\underline{\Psi}^{(j)})$ for $1\leq i,j\leq 3$.

%%%%%%%%%%%%%%%%%%%%%%%%%%%%%%%%%%%%%%%%%%%%%%%%
\begin{lem}
Let $n=3$. For the simple reflections $s_1$ and $s_2$ of $S_3$ we have
\begin{equation}\label{toprove1}
\begin{split}
\widetilde{M}^{\pi_{\mathcal{B},D}}(s_1)&=
\mathcal{R}_{23}(z_1-z_2;\underline{\phi}+\widehat{\underline{\Psi}}(\kappa)_1),\\
\widetilde{M}^{\pi_{\mathcal{B},D}}(s_2)&=
\mathcal{R}_{12}(z_2-z_3;\underline{\phi}+\widehat{\underline{\Psi}}(-\kappa)_3)
\end{split}
\end{equation}
as linear operators on $V^{\otimes 3}$, where
\begin{equation} \label{psi1}
\underline{\Psi}^{(j)}(\alpha):=
\begin{cases}
(-\alpha,0,-\pi\sqrt{-1}\log(p)^{-1})\quad &\hbox{ if }\,\, j=1,\\
(0,-\alpha,-\pi\sqrt{-1}\log(p)^{-1})\quad &\hbox{ if }\,\, j=2,\\
(0,0,\alpha)\quad &\hbox{ if }\,\, j=3.
\end{cases}
\end{equation}
\end{lem}
%%%%%%%%%%%%%%%%%%%%%%%%%%%%%%%%%%%%%%%%%%%%%%%%%
\begin{proof}
The proof of \eqref{toprove1} is a rather long case by case verification which involves computing the action of the left hand side on
the tensor product basis elements using 
Corollary \ref{explicitontensorCOR}. As an example
of the typical arguments, we give here the proof of the first identity in \eqref{toprove1} when acting on the tensor product 
basis vectors $v_1\otimes v_3\otimes v_2$ and $v_2\otimes v_3\otimes v_2$. This will also clarify the subtleties arising from the fact 
that $V^{\otimes 3}$ has multiple principal series blocks,
\[
V^{\otimes 3}=\bigoplus_{\mathbf{r}\in\mathcal{J}_3}(V^{\otimes 3})_{\mathbf{r}}.
\]

Consider the tensor product basis element $v_1\otimes v_3\otimes v_2$. Note that 
\[
v_1\otimes v_3\otimes v_2=\mathbf{v}_\beta\in (V^{\otimes 3})_{(1,1,1)}
\] 
with $\beta:=(1,3,2)\in\mathcal{K}_3[(1,1,1)]$. 
Note that $I^{(1,1,1)}=\emptyset$, $\epsilon^{(1,1,1)}=\emptyset$, 
\begin{equation*}
w_\beta=
\left(\begin{matrix} 1 & 2 & 3\\ 2 & 3 & 1\end{matrix}\right)
\end{equation*}
and 
\begin{equation*}
\begin{split}
\gamma^{(1,1,1)}&=(\eta_3^{(1,1,1)}+\phi_3+2\kappa,\eta_2^{(1,1,1)}+\phi_2-2\kappa,\eta_1^{(1,1,1)}+\phi_1-2\kappa)\\
&=(-2\pi\sqrt{-1}\log(p)^{-1}+\phi_3, -\pi\sqrt{-1}\log(p)^{-1}+\phi_2,-\pi\sqrt{-1}\log(p)^{-1}+\phi_1).
\end{split}
\end{equation*}
Consequently,
\[
\gamma_{w_\beta^{-1}(2)}^{(1,1,1)}-\gamma_{w_\beta^{-1}(3)}^{(1,1,1)}=\gamma_1^{(1,1,1)}-\gamma_2^{(1,1,1)}=\phi_3-\phi_2-\pi\sqrt{-1}\log(p)^{-1}.
\]
Hence Corollary \ref{explicitontensorCOR}{\bf (b)} gives
\begin{equation*}
\begin{split}
\widetilde{M}^{\pi_{\mathcal{B},D}}(s_1)(v_1\otimes v_3\otimes v_2)&=\widetilde{M}^{\pi_{\mathcal{B},D}}(s_1)\mathbf{v}_\beta\\
&=A^{\phi_3-\phi_2-\pi\sqrt{-1}\log(p)^{-1}}(z_1-z_2)\mathbf{v}_\beta-B^{\phi_3-\phi_2-\pi\sqrt{-1}\log(p)^{-1}}(z_1-z_2)\mathbf{v}_{s_2\beta}\\
&=v_1\otimes \mathcal{R}(z_1-z_2;\underline{\phi}+\underline{\Psi}^{(1)}(\kappa))\bigl(v_3\otimes v_2\bigr),
\end{split}
\end{equation*}
which proves the first equality of \eqref{toprove1} when applied to $v_1\otimes v_3\otimes v_2$.

As a second example, we consider the validity of the first equality of \eqref{toprove1} when applied to $v_2\otimes v_3\otimes v_2=\mathbf{v}_\alpha\in (V^{\otimes 3})_{(0,2,1)}$,
where $\alpha:=(2,3,2)\in\mathcal{K}_3[(0,2,1)]$. This time we have $I^{(0,2,1)}=\{2\}$, $\epsilon^{(0,2,1)}=\{+\}$,
\begin{equation*}
w_\alpha=
\left(\begin{matrix} 1 & 2 & 3\\ 2 & 1 & 3\end{matrix}\right)
\end{equation*}
and 
\begin{equation*}
\begin{split}
\gamma^{(0,2,1)}&=(\eta_3^{(0,2,1)}+\phi_3+2\kappa,\eta_2^{(0,2,1)}+\phi_2-2\kappa,\eta_2^{(0,2,1)}+\phi_2-4\kappa)\\
&=(-2\pi\sqrt{-1}\log(p)^{-1}+\phi_3, -\pi\sqrt{-1}\log(p)^{-1}+\phi_2+\kappa,-\pi\sqrt{-1}\log(p)^{-1}+\phi_2-\kappa).
\end{split}
\end{equation*}
Hence
\[
\gamma_{w_\alpha^{-1}(2)}^{(0,2,1)}-\gamma_{w_\alpha^{-1}(3)}^{(0,2,1)}=\gamma_1^{(0,2,1)}-\gamma_3^{(0,2,1)}=\phi_3-\phi_2-\pi\sqrt{-1}\log(p)^{-1}+\kappa.
\]
Therefore, Corollary \ref{explicitontensorCOR}{\bf (b)} gives
\begin{equation*}
\begin{split}
\widetilde{M}^{\pi_{\mathcal{B},D}}(s_1)(v_2\otimes v_3\otimes &v_2)=\widetilde{M}^{\pi_{\mathcal{B},D}}(s_1)\mathbf{v}_\alpha\\
=&A^{\phi_3-\phi_2-\pi\sqrt{-1}\log(p)^{-1}+\kappa}(z_1-z_2)\mathbf{v}_\alpha-B^{\phi_3-\phi_2-\pi\sqrt{-1}\log(p)^{-1}+\kappa}(z_1-z_2)\mathbf{v}_{s_2\alpha}\\
=&v_2\otimes \mathcal{R}(z_1-z_2;\underline{\phi}+\underline{\Psi}^{(2)}(\kappa))\bigl(v_3\otimes v_2\bigr),
\end{split}
\end{equation*}
which proves the first equality of \eqref{toprove1} when applied to $v_2\otimes v_3\otimes v_2$. All other cases can be checked by a similar computation.
\end{proof}
%%%%%%%%%%%%%%%%%%%%%%%%%%%%%%%%%%%%%%

%%%%%%%%%%%%%%%%%%%%%%%%%%%%%%%%%%%%%%%%%%%%%%%%%%%
\begin{cor}\label{lemMAIN}
The linear operator $\mathcal{R}(x;\underline{\phi}): V^{\otimes 2}\rightarrow V^{\otimes 2}$ satisfies 
\begin{equation}\label{dynqYBcomp}
\begin{split}
\mathcal{R}_{12}(x;\underline{\phi}+\widehat{\underline{\Psi}}(-\kappa)_3)&\mathcal{R}_{23}(x+y;\underline{\phi}+\widehat{\underline{\Psi}}(\kappa)_1)
\mathcal{R}_{12}(y;\underline{\phi}+\widehat{\underline{\Psi}}(-\kappa)_3)=\\
&=\mathcal{R}_{23}(y;\underline{\phi}+\widehat{\underline{\Psi}}(\kappa)_1)\mathcal{R}_{12}(x+y;\underline{\phi}+\widehat{\underline{\Psi}}(-\kappa)_3)
\mathcal{R}_{23}(x;\underline{\phi}+\widehat{\underline{\Psi}}(\kappa)_1)
\end{split}
\end{equation}
as linear operators on $V^{\otimes 3}$.
\end{cor}
%%%%%%%%%%%%%%%%%%%%%%%%%%%%%%%%%%%%%%%%%%%%%%%%%%
\begin{proof}
The braid type relation \eqref{dynqYBcomp} is a direct consequence of \eqref{toprove1} in view of the cocycle property 
of the modified monodromy cocycle $\{\widetilde{\mathbb{M}}^{\pi_{\mathcal{B},D}}(u)\}_{u\in S_n}$
(cf. \eqref{YBconnection} for the unmodified monodromy cocycle).
\end{proof}
%%%%%%%%%%%%%%%%%%%%%%%%%%%%%%%%%%%%%%%%%%%%%%%%%%
We are now ready to obtain the proof of Theorem \ref{mainTHMfirst}. It suffices to show that $\mathcal{R}(x;\underline{\phi})$ is satisfying the 
dynamical quantum Yang-Baxter
equation \eqref{dynqYBfirst} in braid-like form. We derive it as consequence of \eqref{dynqYBcomp}.

First of all, replacing $\underline{\phi}$ in \eqref{dynqYBcomp} by
$\underline{\phi}+(0,0,\pi\sqrt{-1}\log(p)^{-1})$ we conclude that 
\begin{equation}\label{dynqYBPhi}
\begin{split}
\mathcal{R}_{12}(x;\underline{\phi}+\widehat{\underline{\Phi}}(-\kappa)_3)&\mathcal{R}_{23}(x+y;\underline{\phi}+\widehat{\underline{\Phi}}(\kappa)_1)
\mathcal{R}_{12}(y;\underline{\phi}+\widehat{\underline{\Phi}}(-\kappa)_3)=\\
&=\mathcal{R}_{23}(y;\underline{\phi}+\widehat{\underline{\Phi}}(\kappa)_1)\mathcal{R}_{12}(x+y;\underline{\phi}+\widehat{\underline{\Phi}}(-\kappa)_3)
\mathcal{R}_{23}(x;\underline{\phi}+\widehat{\underline{\Phi}}(\kappa)_1)
\end{split}
\end{equation}
with respect to the shift vectors 
\begin{equation} \label{psi2}
\underline{\Phi}^{(j)}(\alpha):=
\begin{cases}
(-\alpha,0,0)\quad &\hbox{ if }\,\, j=1,\\
(0,-\alpha,0)\quad &\hbox{ if }\,\, j=2,\\
(0,0,\alpha+\pi\sqrt{-1}\log(p)^{-1})\quad &\hbox{ if }\,\, j=3.
\end{cases}
\end{equation}
Now note that the 
dynamical quantum Yang-Baxter equation \eqref{dynqYBfirst} is equivalent to the equation
\begin{equation}\label{dynqYBfirstXi}
\begin{split}
\mathcal{R}_{12}(x;\underline{\phi}+\widehat{\underline{\Xi}}(-\kappa)_3)&\mathcal{R}_{23}(x+y;\underline{\phi}+\widehat{\underline{\Xi}}(\kappa)_1)
\mathcal{R}_{12}(y;\underline{\phi}+\widehat{\underline{\Xi}}(-\kappa)_3)=\\
&=\mathcal{R}_{23}(y;\underline{\phi}+\widehat{\underline{\Xi}}(\kappa)_1)\mathcal{R}_{12}(x+y;\underline{\phi}+\widehat{\underline{\Xi}}(-\kappa)_3)
\mathcal{R}_{23}(x;\underline{\phi}+\widehat{\underline{\Xi}}(\kappa)_1)
\end{split}
\end{equation}
with shift vectors
\begin{equation*} 
\underline{\Xi}^{(j)}(\alpha):=
\begin{cases}
(-\alpha,0,0)\quad &\hbox{ if }\,\, j=1,\\
(0,-\alpha,0)\quad &\hbox{ if }\,\, j=2,\\
(0,0,\alpha)\quad &\hbox{ if }\,\, j=3.
\end{cases}
\end{equation*}
So it remains to show that the $\pi\sqrt{-1}\log(p)^{-1}$ term in $\Phi_3^{(3)}(\alpha)$ may be omitted in the equation \eqref{dynqYBPhi}. Acting by both sides of \eqref{dynqYBPhi} on a pure tensor $v_i\otimes v_j\otimes v_k$, the resulting equation involves the shift $\Phi_3^{(3)}(\alpha)$ only if two of the indices $i,j,k$ are equal  to $3$.
In case $(i,j,k)\in\{(1,3,3), (3,1,3), (3,3,1)\}$ the dependence on the dynamical parameters is a dependence on
\[
(\phi_1+\Phi_1^{(3)}(\pm\kappa))-(\phi_3+\Phi_3^{(3)}(\pm\kappa))=\phi_1-\phi_3\mp\kappa-\pi\sqrt{-1}\log(p)^{-1}.
\]
Thus replacing $\phi_3$ by $\phi_3-\pi\sqrt{-1}\log(p)^{-1}$, it follows that the equation is equivalent to the equation with $\Phi_3^{(3)}(\alpha)$ omitted.
A similar argument applies to the case $(i,j,k)\in\{(2,3,3), (3,2,3), (3,3,2)\}$. This proves \eqref{dynqYBfirstXi} and thus completes the proof of Theorem \ref{mainTHMfirst}.

%%%%%%%%%%%%%%%%%%%%%%%%%%%%%%%%%%%%%%%%%%%%%%%%%
\section{Appendix}\label{App}
%%%%%%%%%%%%%%%%%%%%%%%%%%%%%%%%%%%%%%%%%%%%%%%%%%
The computation of the connection matrices of quantum affine KZ equations associated to principal series modules in \cite[\S 3]{S1} 
only deal with principal series modules $M^{I,\epsilon}(\gamma)$ with $\epsilon_i=+$ for all $i$. We describe here the extension of the results in \cite[\S 3]{S1} to 
include the case of signs $\epsilon_i$ ($i\in I$) such that $\epsilon_i=\epsilon_j$ if $s_i$ and $s_j$ are in the same conjugacy class of $S_{n,I}$. 
Following \cite{S1} we will discuss it in the general context of arbitrary root data.  

We forget for the moment the notations and conventions from the previous sections and freely use the notations from \cite[\S 3.1]{S1}.
In case of $\textup{GL}(n)$ initial data, these notations slightly differ from the notations of the previous sections (for instance, our present parameter $p$ corresponds to $q$ in \cite{S1}).  At the end of the appendix we will explicitly translate the results in this appendix to the setting and conventions of this paper.

Fix a choice of initial data $(R_0,\Delta_0,\bullet, \Lambda,\widetilde{\Lambda})$ (see \cite[\S 3.1]{S1} for more details)
and a subset $I\subseteq\{1,\ldots,n\}$. Write $W_0$ for the finite Weyl group associated
to $R_0$ and $W_{0,I}\subseteq W_0$ for the parabolic subgroup generated by $s_i$ ($i\in I$).
Fix a $\#I$-tuple $\epsilon=(\epsilon_i)_{i\in I}$
of signs  such that $\epsilon_i=\epsilon_j$ if $s_i$ and $s_j$ are conjugate in $W_{0,I}$. We write
\[
E_{\mathbb{C}}^{I,\epsilon}:=\{\gamma\in E_{\mathbb{C}} \,\, | \,\, (\widetilde{\alpha}_i,\gamma)=\epsilon_i(\widetilde{\kappa}_{\widetilde{\alpha}_i}+
\widetilde{\kappa}_{2\widetilde{\alpha}_i})\quad \forall\, i\in I \, \}
\]
with $E_{\mathbb{C}}$ the complexification of the ambient Euclidean space $E$ of the root system $R_0$.
The definition \cite[Def. 3.3]{S1} of the principal series module of the (extended) affine Hecke algebra $H_n(\kappa)$ 
now generalises as follows,
\[
M^{I,\epsilon}(\gamma):=\textup{Ind}_{H_I(\kappa)}^{H(\kappa)}\bigl(\mathbb{C}_{\chi^{I,\epsilon}_{\gamma}}\bigr),\qquad \gamma\in E_{\mathbb{C}}^{I,\epsilon},
\]
with $\chi^{I,\epsilon}_{\gamma}: H_I(\kappa)\rightarrow\mathbb{C}$ being the linear character defined by
\begin{equation*}
\begin{split}
\chi^{I,\epsilon}_\gamma(T_i)&:=\epsilon_iq^{-\epsilon_i\kappa_i},\qquad\,\,\, i\in I,\\
\chi^{I,\epsilon}_\gamma(Y^\nu)&:=q^{-(\nu,\gamma)},\qquad\quad \nu\in\widetilde{\Lambda}.
\end{split}
\end{equation*}
We write $M(\gamma)$ for $M^{I,\epsilon}(\gamma)$ when $I=\emptyset$.

We now generalise the two natural bases of the principal series modules. Fix generic $\gamma\in E_{\mathbb{C}}^{I,\epsilon}$. 
For $w\in W_0$ set
\[
v_w^{I,\epsilon}(\gamma):=T_w\otimes_{H_I(\kappa)}\mathbb{C}_{\chi_\gamma^{I,\epsilon}}\in M^{I,\epsilon}(\gamma).
\]
Note that $v_w^{I,\epsilon}(\gamma)=\chi_v^{I,\epsilon}(T_v)v_u^{I,\epsilon}(\gamma)$ if $w=uv$ with $u\in W_0^I$ and $v\in W_{0,I}$.
We write $v_w(\gamma)$ for $v_w^{I,\epsilon}(\gamma)$ if $I=\emptyset$.
Let $\phi_\gamma^{I,\epsilon}: M(\gamma)\twoheadrightarrow M^{I,\epsilon}(\gamma)$ be the canonical intertwiner mapping $v_w(\gamma)$
to $v_w^{I,\epsilon}(\gamma)$ for $w\in W_0$. Then \cite[Prop. 3.4]{S1} is valid for $M^{I,\epsilon}(\gamma)$, with the unnormalised elements 
$b_w^{unn,I}(\gamma)$ replaced by
\[
b_w^{unn,I,\epsilon}(\gamma):=\phi_\gamma^{I,\epsilon}\bigl(A_w^{unn}(\gamma)v_e(w^{-1}\gamma)\bigr),\qquad w\in W_0.
\]
Indeed, as in the proof of \cite[Prop. 3.4]{S1}, one can show by a direct computation that
\[
\phi_\gamma^{I,\epsilon}\bigl(A_{s_i}^{unn}(\gamma)v_\tau(s_i\gamma)\bigr)=0,\qquad \,\forall\, \tau\in W_0
\]
if $i\in I$ and $\epsilon_i\in\{\pm\}$ (despite the fact that the term $D_{\widetilde{\alpha}_i}(\gamma)$ appearing in the
proof of \cite[Prop. 3.4]{S1} is no longer zero when $i\in I$ and $\epsilon_i=-$). Now in the same way as in \cite[\S 3.2]{S1}, the normalised
basis $\{b_{\sigma^{-1}}^{I,\epsilon}(\gamma)\}_{\sigma\in W_0^I}$ of $M^{I,\epsilon}(\gamma)$ can be defined by
\[
b_{\sigma^{-1}}^{I,\epsilon}(\gamma):=D_{\sigma^{-1}}(\gamma)^{-1}b_{\sigma^{-1}}^{unn,I,\epsilon}(\gamma),\qquad \sigma\in W_0^I,
\]
see \cite[Cor 3.6]{S1}.

Following \cite[\S 3.4]{S1} we write, for a finite dimensional affine Hecke algebra module $L$, 
$\nabla^L$ for the action of the extended affine Weyl group $W$ on the space of $L$-valued meromorphic functions on $E_{\mathbb{C}}$
given by
\[
\bigl(\nabla^L(w)f\bigr)(\mathbf{z})=C_w^L(\mathbf{z})f(w^{-1}\mathbf{z}),\qquad w\in W
\]
for the explicit $W$-cocycle $\{C_w^L\}_{w\in W}$ as given by \cite[Thm. 3.7]{S1}. Cherednik's \cite{C} quantum affine KZ equations then read
\[
\nabla^L(\tau(\lambda))f=f\qquad \forall \lambda\in\widetilde{\Lambda}, 
\]
see \cite[(3.7)]{S1} in the present notations. In the limit $\Re\bigl((\alpha,\mathbf{z})\bigr)\rightarrow-\infty$ for all $\alpha\in R_0^+$,
the transport operators $C_{\tau(\lambda)}^L(\mathbf{z})$ 
tend to $\pi(\widetilde{Y}^\lambda)$ for $\lambda\in\widetilde{\Lambda}$, where $\pi$ is the representation map of $L$ and
\[
\widetilde{Y}^\lambda:=q^{-(\rho,\lambda)}T_{w_0}Y^{w_0\lambda}T_{w_0}^{-1}
\]
with $w_0\in W_0$ the longest Weyl group element.

An $F$-basis of solutions of the quantum affine KZ equations  
\[
\bigl(\nabla^{M^{I,\epsilon}(\gamma)}(\tau(\lambda))f\bigr)(\mathbf{z})
f(\mathbf{z}),\qquad \forall\,\lambda\in\widetilde{\Lambda}
\]
for $M^{I,\epsilon}(\gamma)$-valued meromorphic functions $f(\mathbf{z})$ in $\mathbf{z}\in E_{\mathbb{C}}$ (see \cite[Def. 3.8]{S1}) is given by
\[
\Phi_{\sigma^{-1}}^{I,\epsilon}(\mathbf{z},\gamma):=\phi_\gamma^{I,\epsilon}\bigl(A_{\sigma^{-1}}(\gamma)\phi_{\sigma\gamma}^{\mathcal{V}}(
\Phi(\mathbf{z},\sigma\gamma))\bigr)\qquad \sigma\in W_0^I
\]
for generic $\gamma\in E_{\mathbb{C}}^{I,\epsilon}$,
where we freely used the notations from \cite[\S 3]{S1} (in particular, $\phi_{\sigma\gamma}^{\mathcal{V}}$ is the linear isomorphism from
$\mathcal{V}=\bigoplus_{w\in W_0}\mathbb{C}v_w$ onto $M(\sigma\gamma)$  mapping $v_w$ to $v_w(\sigma\gamma)$ for $w\in W_0$,
and $\Phi(\mathbf{z},\gamma)$ is the asymptotically free solution of the bispectral quantum KZ equations, defined in \cite[Thm. 3.10]{S1}).
The characterising asymptotic behaviour of $\Phi_{\sigma^{-1}}^{I,\epsilon}(\mathbf{z},\gamma)$ ($\sigma\in W_0^I$)  is
\begin{equation}\label{asympgeneral}
\Phi_{\sigma^{-1}}^{I,\epsilon}(\mathbf{z};\gamma)=q^{(w_0\rho-w_0\sigma\gamma,\mathbf{z})}\sum_{\alpha\in Q_+}
\Gamma_\sigma^{I,\epsilon,\gamma}(\alpha)q^{-(\alpha,\mathbf{z})}
\end{equation}
if $\Re\bigl((\alpha,\mathbf{z})\bigr)\ll 0$ for all $\alpha\in R_0^+$, with $Q_+=\mathbb{Z}_{\geq 0}R_0^+$ and with leading coefficient
\begin{equation}\label{lc}
\begin{split}
\widetilde{b}_\sigma:=&\Gamma_\sigma^{I,\epsilon,\gamma}(0)=
\textup{cst}_\sigma^\gamma\pi^{I,\epsilon}_\gamma(T_{w_0})b_{\sigma^{-1}}^{I,\epsilon}(\gamma),\\
\textup{cst}_\sigma^\gamma:=&
\frac{q^{(\widetilde{\rho},\rho-\sigma\gamma)}}{\widetilde{\mathcal{S}}(\sigma\gamma)}
\Bigl(\prod_{\alpha\in R_0^+}\bigl(q_\alpha^2q^{-2(\widetilde{\alpha},\sigma\gamma)};q_\alpha^2\bigr)_{\infty}\Bigr),
\end{split}
\end{equation}
where $\pi_\gamma^{I,\epsilon}$ is the representation map of $M^{I,\epsilon}(\gamma)$, see \cite[Prop. 3.13]{S1}. Note that
\[
\pi_{\gamma}^{I,\epsilon}(\widetilde{Y}^\lambda)\widetilde{b}_\sigma=q^{-(\lambda,\rho+w_0\sigma\gamma)}\widetilde{b}_\sigma
\qquad \forall\, \lambda\in\widetilde{\Lambda}.
\]

For generic $\gamma\in E_{\mathbb{C}}^{I,\epsilon}$, there exists unique $m_{\tau_1,\tau_2}^{I,\epsilon,\sigma}(\cdot,\gamma)\in F$
($\sigma\in W_0$, $\tau_1,\tau_2\in W_0^I$) such that
\begin{equation}\label{mdef}
\nabla^{M^{I,\epsilon}(\gamma)}(\sigma)\Phi_{\tau_2^{-1}}^{I,\epsilon}(\cdot,\gamma)=
\sum_{\tau_1\in W_0^I}m_{\tau_1,\tau_2}^{I,\epsilon,\sigma}(\cdot,\gamma)\Phi_{\tau_1^{-1}}^{I,\epsilon}(\cdot,\gamma)
\end{equation}
for all $\sigma\in W_0$ and $\tau_1\in W_0^I$. The connection matrices
\[
M^{I,\epsilon,\sigma}(\cdot,\gamma):=\bigl(m_{\tau_1,\tau_2}^{I,\epsilon,\sigma}(\cdot,\gamma)\bigr)_{\tau_1,\tau_2\in W_0^I},\qquad \sigma\in W_0
\]
satisfy the cocycle properties $M^{I,\epsilon,\sigma\sigma^\prime}(\mathbf{z},\gamma)=M^{I,\epsilon,\sigma}(\mathbf{z},\gamma)
M^{I,\epsilon,\sigma^\prime}(\sigma^{-1}\mathbf{z},\gamma)$ for $\sigma,\sigma^\prime\in W_0$, and 
$M^{I,\epsilon,e}(\mathbf{z},\gamma)=\textup{Id}$. Now \cite[Thm. 3.15]{S1} generalises as follows.

For $i\in \{1,\ldots,n\}$ we write $i^*\in\{1,\ldots,n\}$ for the index such that $\alpha_{i^*}=-w_0\alpha_i$, where $w_0\in W_0$ is the longest Weyl group element.
The {\it elliptic $c$-function} is defined by
\begin{equation}\label{cfunctiongeneral}
c_{\alpha}(x):=
\frac{\theta(a_\alpha q^{x}, b_\alpha q^{x}, c_{\alpha}q^{x}, d_{\alpha}q^{x};q_{\alpha}^2)}
{\theta(q^{2x};q_{\alpha}^2)}q^{\frac{1}{\mu_{\alpha}}(\kappa_{\alpha}+\kappa_{\alpha^{(1)}})x}
\end{equation}
for $\alpha\in R_0$, where $\{a_\alpha,b_\alpha,c_\alpha,d_\alpha\}$ are the Askey-Wilson parameters, see \cite[\S 3.1]{S1}. 

%%%%%%%%%%%%%%%%%%%%%%%%%%%%%%%%%%%%%%%%%%%%%%%%%%%%%%%%
\begin{thm}\label{mainthmgeneral}
Fix a generic $\gamma\in E_{\mathbb{C}}^{I,\epsilon}$ such that $q^{2(\widetilde{\beta},\gamma)}\not\in q_\beta^{2\mathbb{Z}}$ 
for all $\beta\in R_0$. Let $\tau_2\in W_0^I$ and $i\in\{1,\ldots,n\}$. If $s_{i^*}\tau_2\not\in W_0^I$ then
\[
m_{\tau_1,\tau_2}^{I,\epsilon,s_i}(\mathbf{z},\gamma)=\delta_{\tau_1,\tau_2}\epsilon_{i^*_{\tau_2}}\,\frac{c_{\alpha_i}((\alpha_i,\mathbf{z}))}
{c_{\alpha_i}(\epsilon_{i^*_{\tau_2}}(\alpha_i,\mathbf{z}))},\qquad \forall\, \tau_1\in W_0^I,
\]
with $i_{\tau_2}^*\in I$ such that $\alpha_{i_{\tau_2}^*}=\tau_2^{-1}(\alpha_{i^*})$. If $s_{i^*}\tau_2\in W_0^I$ then $m_{\tau_1,\tau_2}^{I,\epsilon,s_i}(\cdot,\gamma)\equiv 0$
if $\tau_1\not\in\{\tau_2,s_{i^*}\tau_2\}$ while
\begin{equation*}
\begin{split}
m_{\tau_2,\tau_2}^{I,\epsilon,s_i}(\mathbf{z},\gamma)&=\frac{\mathfrak{e}_{\alpha_i}((\alpha_i,\mathbf{z}),(\widetilde{\alpha}_{i^*},\tau_2\gamma))-
\widetilde{\mathfrak{e}}_{\alpha_i}((\widetilde{\alpha}_{i^*},\tau_2\gamma),(\alpha_i,\mathbf{z}))}{\widetilde{\mathfrak{e}}_{\alpha_i}((\widetilde{\alpha}_{i^*},\tau_2\gamma),
-(\alpha_i,\mathbf{z}))},\\
m_{s_{i^*}\tau_2,\tau_2}^{I,\epsilon,s_i}(\mathbf{z},\gamma)&=\frac{\mathfrak{e}_{\alpha_i}((\alpha_i,\mathbf{z}),-(\widetilde{\alpha}_{i^*},\tau_2\gamma))}
{\widetilde{\mathfrak{e}}_{\alpha_i}((\widetilde{\alpha}_{i^*},\tau_2\gamma),-(\alpha_i,\mathbf{z}))},
\end{split}
\end{equation*}
with the functions $\mathfrak{e}_\alpha(x,y)$ and $\widetilde{\mathfrak{e}}_\alpha(x,y)$ given by 
\begin{equation*}
\begin{split}
\mathfrak{e}_\alpha(x,y)&:=
q^{-\frac{1}{2\mu_\alpha}(\kappa_\alpha+\kappa_{2\alpha}-x)(\kappa_\alpha+\kappa_{\alpha^{(1)}}-y)}
\frac{\theta\bigl(\widetilde{a}_\alpha q^y,\widetilde{b}_\alpha q^y, \widetilde{c}_\alpha q^y, d_\alpha q^{y-x}/\widetilde{a}_\alpha;
q_\alpha^2\bigr)}{\theta\bigl(q^{2y},d_\alpha q^{-x};q_\alpha^2\bigr)},\\
\widetilde{\mathfrak{e}}_\alpha(x,y)&:=
q^{-\frac{1}{2\mu_\alpha}(\kappa_\alpha+\kappa_{\alpha^{(1)}}-x)(\kappa_\alpha+\kappa_{2\alpha}-y)}
\frac{\theta\bigl(a_\alpha q^y,b_\alpha q^y,c_\alpha q^y,\widetilde{d}_\alpha q^{y-x}/a_\alpha;
q_\alpha^2\bigr)}{\theta\bigl(q^{2y},\widetilde{d}_\alpha q^{-x};q_\alpha^2\bigr)}.
\end{split}
\end{equation*}
Here $\{\widetilde{a}_\alpha,\widetilde{b}_\alpha,\widetilde{c}_\alpha,\widetilde{d}_\alpha\}$ are the dual Askey-Wilson parameters, see \cite[\S 3.1]{S1}.
\end{thm}
\begin{proof}
Repeating the proof of \cite[Thm. 3.15]{S1} in the present generalised setup we directly obtain the result for $\tau_2\in W_0^I$ satisfying $s_{i^*}\tau_2\in W_0^I$. If $s_{i^*}\tau_2\not\in W_0^I$
then the proof leads to the expression
\[
m_{\tau_1,\tau_2}^{I,\epsilon,s_i}(\mathbf{z},\gamma)=\delta_{\tau_1,\tau_2}n_{\tau_2,\tau_2}^{s_i}(\mathbf{z},\gamma),\qquad \tau_1\in W_0^I
\]
with
\[
n_{\tau_2,\tau_2}^{s_i}(\mathbf{z},\gamma)=\frac{\mathfrak{e}_{\alpha_i}((\alpha_i,\mathbf{z}),(\widetilde{\alpha}_{i^*_{\tau_2}},\gamma))-
\widetilde{\mathfrak{e}}_{\alpha_i}((\widetilde{\alpha}_{i^*_{\tau_2}},\gamma),(\alpha_i,\mathbf{z}))}
{\widetilde{\mathfrak{e}}_{\alpha_i}((\widetilde{\alpha}_{i^*_{\tau_2}},\gamma),-(\alpha_i,\mathbf{z}))}.
\]
So it suffices to show that
\begin{equation*}
n_{\tau_2,\tau_2}^{s_i}(\mathbf{z},\gamma)=
\begin{cases}
1\qquad &\hbox{ if }\epsilon_{i^*_{\tau_2}}=+,\\
-\frac{c_{\alpha_i}((\alpha_i,\mathbf{z}))}
{c_{\alpha_i}(-(\alpha_i,\mathbf{z}))}\qquad &\hbox{ if }\epsilon_{i^*_{\tau_2}}=-
\end{cases}
\end{equation*}
if $s_{i^*}\tau_2\not\in W_0^I$. The case $\epsilon_{i^*_{\tau_2}}=+$ is proved in \cite[Thm. 3.15]{S1} by applying a nontrivial theta-function identity.
If $\epsilon_{i^*_{\tau_2}}=-$ then 
\begin{equation*}
\mathfrak{e}_{\alpha_i}((\alpha_i,\mathbf{z}),(\widetilde{\alpha}_{i^*_{\tau_2}},\gamma))=
\mathfrak{e}_{\alpha_i}((\alpha_i,\mathbf{z}),-\widetilde{\kappa}_{\widetilde{\alpha}_i}-\widetilde{\kappa}_{2\widetilde{\alpha}_i})=0,
\end{equation*}
hence
\begin{equation*}
\begin{split}
n_{\tau_2,\tau_2}^{s_i}(\mathbf{z},\gamma)&=-\frac{\widetilde{\mathfrak{e}}_{\alpha_i}((\widetilde{\alpha}_{i^*_{\tau_2}},\gamma),(\alpha_i,\mathbf{z}))}
{\widetilde{\mathfrak{e}}_{\alpha_i}((\widetilde{\alpha}_{i^*_{\tau_2}},\gamma),-(\alpha_i,\mathbf{z}))}\\
&=-\frac{\widetilde{\mathfrak{e}}_{\alpha_i}(-\widetilde{\kappa}_{\widetilde{\alpha}_i}-\widetilde{\kappa}_{2\widetilde{\alpha}_i},(\alpha_i,\mathbf{z}))}
{\widetilde{\mathfrak{e}}_{\alpha_i}(-\widetilde{\kappa}_{\widetilde{\alpha}_i}-\widetilde{\kappa}_{2\widetilde{\alpha}_i},-(\alpha_i,\mathbf{z}))}=
-\frac{c_{\alpha_i}((\alpha_i,\mathbf{z}))}
{c_{\alpha_i}(-(\alpha_i,\mathbf{z}))},
\end{split}
\end{equation*}
where the last equality follows by a direct computation.
\end{proof}
%%%%%%%%%%%%%%%%%%%%%%%%%%%%%%%%%%%%%%%%%%%%%%%%%%%%%%%%
In this paper we have used this general result in the special case of the $\textup{GL}(n)$ initial data $(R_0,\Delta_0,\bullet,\Lambda,\widetilde{\Lambda})$, with
root system $R_0=\{e_i-e_j\}_{1\leq i\not=j\leq n}\subset\mathbb{R}^n=:E$ of type $A_{n-1}$ (here $\{e_i\}_{i=1}^n$ denotes 
the standard orthonormal basis of $\mathbb{R}^n$),
with $\Delta=\{\alpha_1,\ldots,\alpha_{n-1}\}=\{e_1-e_2,\ldots, e_{n-1}-e_n\}$, with
$\bullet=u$ (hence $\mu_\alpha=1$ and $\widetilde{\alpha}=\alpha$ for all $\alpha\in R_0$), 
and with lattices $\Lambda=\mathbb{Z}^n=\widetilde{\Lambda}$. In this case $i^*=n-i$ for $i\in\{1,\ldots,n-1\}$ and the multiplicity function $\kappa$ is constant and equal to
the dual multiplicity function $\widetilde{\kappa}$. The corresponding Askey-Wilson parameters, which coincide in this case with the dual Askey-Wilson parameters, are 
independent of $\alpha\in R_0$ and are given by
\[
\{a,b,c,d\}=\{q^{2\kappa},-1,q^{1+2\kappa},-q\}.
\]
Then the elliptic $c$-function \eqref{cfunctiongeneral} reduces to 
\[
c_\alpha(x)=q^{2\kappa x}\frac{\theta(q^{2\kappa+x};q)}{\theta(q^x;q)}
\]
for $\alpha\in R_0$, and 
\begin{equation*}
\begin{split}
m_{\tau_2,\tau_2}^{I,\epsilon,s_i}(\mathbf{z},\gamma)&=\frac{\theta(q^{2\kappa},q^{(\alpha_{i^*},\tau_2\gamma)-
(\alpha_i,\mathbf{z})};q)}{\theta(q^{(\alpha_{n-i},\tau_2\gamma)},q^{2\kappa-(\alpha_i,\mathbf{z})};q)}
q^{(2\kappa-(\widetilde{\alpha}_{n-i},\tau_2\gamma))(\alpha_i,\mathbf{z})},\\
m_{s_{n-i}\tau_2,\tau_2}^{I,\epsilon,s_i}(\mathbf{z},\gamma)&=\frac{\theta(q^{2\kappa-(\alpha_{n-i},\tau_2\gamma)},
q^{-(\alpha_i,\mathbf{z})};q)}{\theta(q^{2\kappa-(\alpha_i,\mathbf{z})},q^{-(\alpha_{n-i},\tau_2\gamma)};q)}
q^{2\kappa((\alpha_i,\mathbf{z})-(\alpha_{n-i},\tau_2\gamma))}
\end{split}
\end{equation*}
if $\tau_2\in W_0^I$ and $s_{n-i}\tau_2\in W_0^I$ by \cite[(1.9) \& Prop. 1.7]{S0}.

The precise connection to the notation in the main text is given as follows: if $q$ is replaced by $p$ in the above
formulas, then the principal series modules coincide with the ones as defined in Subsection \ref{PSRsection} and the connection matrices $M^{I,\epsilon,w}(\mathbf{z},\gamma)$ become the matrices $\mathbb{M}^{I,\epsilon,w}(\mathbf{z},\gamma)$ as defined by \eqref{Cexpl}.

%%%%%%%%%%%%%%%%%%%%%%%%%%%%%%%%%%%%%%%%%%%%%%%%

\end{document}